\documentclass[10pt]{amsart}
\usepackage{amsmath}
\usepackage{graphicx} % Required for inserting images
\usepackage{listings} % For typesetting (pseudo-) code
\usepackage{hyperref,xcolor}
\usepackage[T1]{fontenc}
\usepackage{pdfsync}
\usepackage{yfonts}      % For biblical-style font (Fraktur)
\usepackage{tcolorbox}   % For creating a commandment-style box
\usepackage{lipsum} 

\newtheorem{theorem}{Theorem}
\newtheorem{conjecture}[theorem]{Conjecture}
\newtheorem{corollary}[theorem]{Corollary}

\newtheorem{lemma}[theorem]{Lemma}

\newtheorem{remark}[theorem]{Remark}
\newtheorem{remarks}[theorem]{Remarks}
\newtheorem{example}[theorem]{Example}

\newcommand{\seqnum}[1]{\href{https://oeis.org/#1}{\rm \underline{#1}}}
\DeclareMathSymbol{\lsim}{\mathord}{symbols}{"18}

\title{On $t$-sumfree sequences}
\author[van Berkel\and Bosma]{Daan van Berkel\and Wieb Bosma}
\email{daan.v.berkel.1980@gmail.com\and bosma@math.ru.nl} 
\date{\today}

\begin{document}
\begin{abstract}
In this paper we explore the world of $t$-sumfree sequences, primarily
for $t=2$ and $t=3$. Starting with $t$ increasing positive integers, build
an infinite sequence by choosing as the next entry the least positive
integer exceeding the current entries that is not the sum of $t$ distinct
previous values. The main
open problem concerns the question of the conjectured ultimate periodicity
of these sequences. We formulate a sufficient condition
to prove that a conjectured preperiod and period are
correct. We also prove periodicity for a considerable subclass
of 2-sumfree and of 3-sumfree sequences. Finally, we present
the results of extensive computations on 3-sumfree sequences.
\end{abstract}

\subjclass[2000]{}
\maketitle

\section{Introduction}
The infinite sequences of increasing positive integers we study in this paper arise as follows: initialize the first $t\geq 1$ values $s_1<s_2<\cdots<s_{t}$, and for $n>t$ let $s_n$ be the smallest
integer exceeding $s_{n-1}$ that is not equal to the sum of
$t$ different previous elements in the sequence; these are {\it 
greedy, strict $t$-sumfree sequences}, which we will refer to in 
what follows in short as {\it $t$-sumfree} sequences. The requirement 
that $s_n$ is the {\it smallest} integer is the greedy aspect, and
that it is not the sum of $t$ {\it different} previous entries is the
strict part. The main new results in this paper, which we will summarize
in this introduction, concern the case $t=3$, and the most
interesting problem is the periodicity question. 

The prototypical example is the sequence $S_{1,2,3}$ for $t=3$,
starting
$$S_{1,2,3}=(1, 2, 3, 4, 5, 13, 14, 15, 25, 26, 27, 37, 38, 48, 49, 50, 60, 61, \ldots);$$
it is  \seqnum{A026471} in the OEIS (Online Encyclopedia of Integer Sequences). Our initial reason for looking at it was to
consider the question of its automaticity and the possibility
to prove statements about this using Walnut \cite{walnut:moussavi}, \cite{walnut:shallit}. It turned out \cite{MMath} that
this sequence is `eventually periodic', and that the same is true for many
alternative starting values (for any $t$), leading to the very 
general conjecture that all $t$-sumfree sequences are `eventually 
periodic'. We hyphenated these statements because they do not make 
sense as stated, since the sequences are strictly increasing. What is 
meant by such an informal
statement is that the corresponding characteristic sequence is
eventually periodic. In the above example this characteristic sequence
$C_{1,2,3}$ starts as
$$(1, 1, 1, 1, 1, 0,0,0,0,0,0,0,1,1,1,0,0,0,0,0,0,0,0,0,1,1,1,0,\ldots),$$
where the final 1 here refers to the entry 27 in $S_{1,2,3}$.
A consequence of such ultimate periodicity of the characteristic sequence
$C$ of $S$
is that there exists a positive integer $m$
such that $S$ modulo $m$ is periodic from a certain entry on, with
period length that we will indicate by $p$ and preperiod length $k$. 
In our example, $C_{1,2,3}$ becomes periodic after
6 ones, with a period of length 23 containing 5 ones:
$$C_{1,2,3}=(1^5,0^7,1,\overline{1^2,0^9,1^3,0^9}).$$
Here the expression $1^5$ is shorthand for $1,1,1,1,1$ etc., and the bar
indicates indefinite repetition of the block.
This implies that $S$, when taken modulo $m=23$ (the sum of the exponents
in the period above), becomes periodic after the
sixth entry, with period length $p=5$ (because of the five ones):
$$S_{1,2,3}\equiv(1,2,3,4,5,13,\overline{14, 15, 25, 26, 27})\bmod 23.$$
Here it is implied in the notation that {\it all} values 
$14+j\cdot 23$ for $j\geq 0$ are included in $S$, and likewise 
for the other residue classes listed under the bar.

Yet another way to express the periodicity properly, is by looking at the
first difference sequence $D_{f,g,h}$ of $S_{f,g,h}$, defined by
$$\big(D_{f,g,h}\big)_{i=1}^\infty=\big(S_{f,g,h}[i+1]-S_{f,g,h}[i]\big)_{i=1}^\infty.$$
Thus
$$D_{1,2,3}=(1,1,1,1,8,1,1,10,1,1,10,1,10,\ldots)=\big(1^4,8,1,\overline{1,10,1^2,10}\big),$$
whose correspondence with $C_{1,2,3}$ will be clear, and from this
the corresponding property for $S_{1,2,3}$ of being ultimately periodic.
In general, $C_{f,g,h}$ is ultimately periodic with preperiod $K$ and period $P$
if and only if $D_{f,g,h}$ is ultimately periodic and there exist integers 
$k\geq 0$, $p\geq 1$ and $m\geq 2$ such that $S_{f,g,h}$ consists of 
entries $a_1<a_2<\cdots<a_k$ and all integers of the form
$b_i+j\cdot m$ for $j\geq 0$, with $1\leq i\leq p$ and $b_1<b_2<\cdots<b_p$.
Here the `modulus' $m$ corresponds to the period length $\#P$ of $C_{f,g,h}$ and 
$k$ equals the number of 1s in the preperiod $K$ of $C_{f,g,h}$, while
$p$ equals the number of 1s in the period $P$ of $C_{f,g,h}$.

Note that the difference sequence $D$ will have the same (pre)period length
$k$ and $p$ as the original $S$. Whenever we mention `the (pre)period length'
of a sequence, we imply that we refer to the minimal values.

Informal statements about periodicity of 2-sumfree sequences $S_{f,g}$ or 3-sumfree sequences $S_{f,g,h}$, should henceforth
be interpreted as proper statements about periodicity of the corresponding characteristic sequence and the sequence of differences.

The type of question we would like to address is this:
is it true that every $t$-sumfree sequence is eventually
periodic in the above sense, and if so, what can be said about
(the dependence of) period length $p$, preperiod size $k$ and modulus $m$ (on $t$ and $s_1, s_2, \ldots, s_{t}$)? Our main results concern the cases $t=2$ and $t=3$. We emphasize
that in our terminology
the sumfree sequences in these theorems are greedy and strict.

In Section \ref{sec:t=2} we prove a more precise version (see Theorem \ref{thm:2subs}) of the following.
\begin{theorem}\label{thm:k2short}
For every $d\geq 1$ and $f\geq d+1$ the 2-sumfree sequence $S_{f,f+d}$ is ultimately periodic.
\end{theorem}
We could be
much more precise, for instance about the period
length and about an extended version with a smaller 
restriction on $f$, but the statements quickly 
become rather technical. 

Fairly extensive
computations do lead us to believe that the following is indeed true.
\begin{conjecture}
For every $g> f\geq 1$ the 2-sumfree sequence $S_{f, g}$ is ultimately periodic.
\end{conjecture}
For much more on the case $t=2$, in particular precise conjectures on periodicity for all pairs $f, g$
and the exact size of period and preperiod lengths, consult
\cite{BB2}. 

For 3-sumfree sequences the situation is less obvious. In Section~\ref{sec:t=3} we prove the following, see Theorem \ref{thm:gend>1} and Example \ref{ex:124}.
\begin{theorem}\label{thm:k=3d>1short}
For every $d\geq 1$, every $f\geq 1$ and every $g\geq\max(f+1,d)$ the 3-sumfree sequence $S_{f,g,g+d}$ is ultimately periodic.
\end{theorem}
Note that for every pair of parameters $(f, d)$, Theorem \ref{thm:k=3d>1short}
leaves only finitely many 3-sumfree sequences $S_{f,g,g+d}$ (values of $g$) that could possibly be not ultimately periodic.

In an initial version of this paper we erroneously included the
family $f=1$, $g=d-1$, $h=2d-1$ for $d\geq 3$ in Theorem \ref{thm:k=3d>1short},
although this is not covered by Theorem \ref{thm:gend>1}. In fact, this
family behaves slightly differently, and although our proof methods
apply here and we are convinced the following result is correct,
we have not written down a full proof and so it is correct to state this
case as a conjecture.
It is just one example of various other parameter families for which
computations of small values readily provide conjectural period and
preperiod lengths.

\begin{conjecture}\label{conj:except}
For every $d\geq 3$ 
the 3-sumfree sequence $S_{1, d-1, 2d-1}$ is ultimately periodic,
and, with the exception of $d=3$ or $d=5$, the preperiod
length is $d+3$ and the period length is $d$ modulo $8d-1$.
\end{conjecture}

We did very extensive experiments to provide evidence for the main
general conjecture.
\begin{conjecture}[folklore]\label{conj:folk}
Every $3$-sumfree sequence $S_{f,g,h}$ is ultimately periodic.
\end{conjecture}
The reservations we have about it, stem from our attempts
to provide computational evidence for
this conjecture and the exceptions mentioned in 
the following theorem that we were able to establish.
\begin{theorem}\label{thm:smallshort}
Among all triples $(f,g,h)$ with $f<g<h$, satisfying $f\leq 25$ and $d=h-g\leq 25$ there are at most 12 exceptions \footnote{But see Remark \ref{rmk:hiro}} 
to the rule that the 3-sumfree sequence $S_{f,g,h}$ is ultimately periodic.
\end{theorem}
For $d\leq 25$ there are 2024 triples $(f, g, g+d)$ with
$f\leq 25$ not covered by
either Theorem \ref{thm:k=3d>1short} or Conjecture
\ref{conj:except}, namely $\frac{(d-3)(d-2)}{2}$ for each
$d$ with $4\leq d\leq 25$. These were the cases we investigated
by computation, and
all but 12 of these were shown to be ultimately periodic (but see Remark \ref{rmk:hiro}).
For the proofs of ultimate periodicity we used an elementary
lemma (Lemma \ref{lemma} in Section \ref{sec:lemma}) that shows how such
proofs follow from a {\it finite} computation.

Section \ref{sec:algo} contains some remarks on the algorithms 
we implemented, and Section \ref{sec:small} contains some 
computational details, such as tables with (parts of) the results
mentioned in Theorem \ref{thm:smallshort} and
the number of terms (at least 500000) computed for the 
12 possible exceptions.

Apart from the theorems above, we have no convincing argument for Conjecture~\ref{conj:folk}.

\section{Related work}\label{sec:related}
There is a considerable literature on sumfree sets and sequences.
In most cases the sequences called sumfree are greedy, 
$2$-sumfree sequences of the non-strict kind: after an initial
segment (sometimes consisting of more than 2 elements) each entry
is required to be the smallest integer exceeding previous entries
that is not the sum of any two previous entries: $a_n>a_{n-1}$ and not $a_n=a_i+a_j$
for $1\leq i, j\leq n-1$ (note the inclusion of $i=j$ here).

In the case of {\it sumfree sets} there is often emphasis on questions
of density: how large can sumfree sets be relative to their universe?
For our sequences we just mention that the natural density 
for $2$-sumfree sequences seems always bounded by $\frac{1}{2}$
(compare Theorem \ref{thm:k2odd} and Theorem \ref{thm:2subs}), while
for $t=3$ we observed only a few sequences with density larger than
$\frac{1}{5}$ (see Example \ref{ex:orig} and Table \ref{short}).

We summarize other relevant results we are aware of in this section.

Apparently (cf.~Guy, \cite{guy} Problem E32) Dickson \cite{dickson}
was the first to coin
questions about sumfree sequences; however, his main focus was on constructing
sequences with the property that every positive integer is the sum of
(at most) $t$ elements in the sequence (allowing repetitions). 

Queneau \cite{queneau} defined $s$-additive sequences, which consist, after an 
initial segment, of those positive integers that can be expressed
in exactly $s$ ways as the sum of two different previous entries. 
For $s=0$ this is the same as our $2$-sumfree sequences. 
His Section 4 concerns this case (he calls those sequences \emph{non-additive}). He considered many explicit cases $S_{f,g}$, and formulates several conjectures,
mostly subsumed by our Theorem \ref{thm:2subs} below. They generally state 
(without proof) that the
resulting sequences are finite unions of residue classes for some
modulus depending on $f,g$, with finitely many exceptions (both additional
and missing entries). Some of these results were summarized by Finch in
\cite{finch} and rephrased in terms of convolution sequences in a Boolean
algebra. In \cite{BB2} we present similar conjectured results for all
2-sumfree sequences.

Cameron noted a natural bijection between the set of 
non-strict, greedy 2-sumfree \emph{sets} and binary sequences $\sigma$.
For given $\sigma$ the positive integers $n$ are tested in
order for inclusion in $S$: when $n$ 
is not the sum of two elements of $S$ the first element of $\sigma$
not used yet decides whether $n$ is added to $S$ (if it is 1) or not.
Cameron \cite{cameron} proved that if $S$ is ultimately periodic
so is $\sigma$, but the converse is not clear. In fact
Calkin and Finch \cite{calkinfinch} did extensive calculations
which left several periodic candidates for $\sigma$ that 
seem to give aperiodic $S$.

\section{A useful lemma}\label{sec:lemma}
The following useful lemma is an adaptation and generalization of a criterion 
formulated by Calkin and Finch \cite{calkinfinch}. 
It shows that ultimate periodicity of 
sumfree sequences can be established by a finite computation. The idea
behind the lemma is simple: if a non-zero block of length $m$ is repeated
sufficiently often in the (characteristic sequence of a) sumfree sequence, 
then it is easy to prove that the block will be repeated indefinitely. 
The problem is that the sequence may start with a preperiod (of unknown
length) and that we need to specify what we mean by `sufficiently often'. 
Note that the initial part of the sequence $S$ may cause
problems, since for $z<s_k$ it is not true that $z\notin S$ 
if and only if $z$ is the $t$-sum of entries of $S$. Also, large
gaps in $S$ correspond to long strings of zeroes in $C$ that will not
persist indefinitely. Finally, although inspired by Lemma 4.4 of \cite{calkinfinch}
that deals with sequences free of sums of two elements, the situation here is
more complicated than simply generalizing the argument to sums of $t$ elements,
since contrary to Calkin and Finch we assume {\it strict} sumfreeness; this forces
a larger number of repeating blocks to ensure ultimate periodicity.

Note that a similar criterion for deciding \emph{non-periodicity} based on a finite
initial segment is lacking.

\begin{lemma}\label{lemma} Let $t\geq 1$.
Suppose that the characteristic sequence $C$ corresponding to a $t$-sumfree
sequence $S$ contains $Q$ consecutive, identical, non-null blocks of length 
$M\geq 1$ from position $K+1$ on, such that 
$$\left(Q-\frac{t(t+1)}{2}\right)M\geq (t-1)K.$$
Then $C$ is ultimately periodic, with a period $E$ of length $m$, 
a divisor of $M$. 
Hence $S$, when taken modulo $m$ is also ultimately periodic, 
with period length $p$ equal to the number of 1s in $E$, with preperiod
length $k$, where $k$ is at most the number of 1s among the first $K$
entries of $C$.
\end{lemma}

\begin{proof}
First observe that the uninteresting case $t=1$ is easily verified
(compare the remark at the beginning of Section \ref{sec:t=2});
so we assume $t\geq 2$.
Suppose that the hypotheses of the Lemma are satisfied, 
but the same blocks of length $M$ do not repeat indefinitely.
Let $v$ the smallest index indicating this non-periodicity, that is,
$v$ is the least index beyond the $Q$ repeating blocks such that $C_v\neq C_{v-M}=C_{v-2M}=\ldots=C_{v-QM}$; then $QM+K+1\leq v$.
Note that $Q\geq \frac{t(t+1)}{2}$.

If $C_v=1$, then $v\in S$ and by hypothesis $v-M\notin S$,
which means that there exist $0<q_1<q_2<\ldots <q_t$ in $S$ such that
$v-M=q_1+q_2+\ldots+q_t$. Now on the one hand
$q_t>K$, because
$$v-M\geq (Q-1)M+K+1\geq \left(Q-\frac{t(t+1)}{2}\right)M+K+1>tK,$$
 so the largest $q_i$ has to exceed $K$; on the other hand
$q_t<v-M$ (since $t>1$), so $q_t+M<v$. Thus $q_t$ is in a repeating block
and $q_t+M$ will be in $S$ like $q_t$, by the minimality assumption on $v$.
This implies that $v=q_1+q_2+\ldots+(q_t+M)$ is a sum of $t$ distinct elements 
of $S$, hence not in $S$, contradicting $C_v=1$. 

If $C_v=0$, then $v\notin S$ and
$v-M, v-2M, \ldots, v-QM$ are all in $S$. In particular, there exist $r_1<r_2<\ldots<r_t\in S$ such that $v=r_1+r_2+\ldots+r_t$.
Let $\mathcal{R}:=\left\{r_1,r_2,\ldots,r_t\right\}$. 
If there exists $r_i\in\mathcal{R}$ such that $r_i-M\in S$ and 
$r_i-M\notin\mathcal{R}$ then
$v-M = r_1 + \ldots + (r_i - M) + \ldots + r_t$ is a sum of $t$ distinct
elements of $S$, contradicting $C_{v-M}=1$. The largest possible sum of elements in
$\mathcal{R}$ such that for all
$r\in \mathcal{R}$ we have $r\in S$ and $r-M\in\mathcal{R}$ or $r-M \notin S$, 
occurs when the $t$ elements of $\mathcal{R}$ exceed $K+M$ but each satisfy
$r_i\leq K+iM$. Then
\[
v = \sum_{i=1}^{t} r_i \leq \sum_{i=1}^t (K + iM) = Kt + M\frac{t(t+1)}{2}\leq QM + K < v
\]
a contradiction.

So neither $C_v = 1$, nor $C_v = 0$ can occur, which means that $C$ must be
ultimately periodic. Clearly, the smallest period length divides the block length $M$, and smallest preperiod of $C$ is not larger than $K$.
\end{proof}

\begin{corollary}\label{corollary}
Let $S$ be a $3$-sumfree sequence.
Suppose that the initial segment of length $9M$ of
the characteristic sequence $C$ of a 3-sumfree sequence $S$
has the property that the 9 blocks of length $M$ of which it is composed
are, with the possible exception of the first block, all identical and non-null,
then $C$ and hence $S$ is ultimately periodic.
\end{corollary}

\begin{proof}
Let $t=3$ and choose $Q=8$ and $K=M$ in Lemma \ref{lemma}.
Since $(Q-\frac{t(t+1)}{2})M=2M\geq(t-1)K$, the result follows.
\end{proof}

\begin{remarks}\rm 
If sums of elements of $S$ with multiplicities greater than 1 are also disallowed
in $S$, (that is, if we would drop the strictness of $S$) then the condition on $Q$ in Lemma \ref{lemma} could be relaxed to
$(Q-t)M\geq (t-1)K$. In that case 6 blocks rather than 9 suffice in Corollary \ref{corollary}.

It is not difficult to translate the corollary
to use the first difference sequence $D$ rather 
than the characteristic sequence $C$. 
It has the advantage that the period lengths of $D$ and $S$ are the same.

\end{remarks}
\section{Algorithmic remarks}\label{sec:algo}
This brief section contains some remarks on (our implementations of)
algorithms to compute sumfree sequences. We did only serious computations
for $t=2$ and $t=3$, mostly using Magma \cite{magma} and Rust \cite{rust}.

There are two issues that deserve some discussion: how to compute
large initial segments of $t$-sumfree sequences, and how to determine
periodicity.

\subsection{Computing the next term}
The most straightforward way to compute $s_{n+1}$, for some $n\geq t$,
once $s_1, s_2, .\ldots, s_n$ have been determined, is to check $s_n+1, s_n+2,\ldots$ in order to see if they can be written
as a sum of $t$ previous elements.

An algorithm to visit each possible combination of $t$ out of $n$ elements
in lexicographical order is not so difficult to devise (compare
\cite{knuth}, Algorithm 7.2.1.3T).

However, we would like a slight variation of the above algorithm, visiting
combinations in order of smallest sum. With the aid of a \emph{heap} data
structure (\cite{knuth}, 5.2.3), it is not difficult to retrieve the combination with
the smallest sum, using that $S$ is strictly increasing.
Here the heap is a priority queue that pops the element with the 
smallest corresponding sum.

To determine the the $n$th element of a $t$-sumfree sequence one has to 
examine at most $\binom{n}{t}$ combinations.
As soon as the next element $s_{n+1}$ is found, it should of course
also be included in the list of possible summands for $s_{n+2}$.

An implementation of a variant of the above algorithm can be found in
\url{https://github.com/fifth-postulate/subsumfree}.

In implementing the suggested algorithm, some optimizations (depending
on $t$) may immediately become apparent; for example, since quickly
$t$ will become small relative to the size of the sumfree sequence, 
it can be beneficial to record the indices of the combinations, 
allowing for a sparse representation.

Alternatively, one could use {\it sieving}:
one would visit all $t$-combi\-nations and remove those from a prepared
list of candidates. This works well if trying to obtain all entries up
to a given bound, but it is not so easy to predict how many terms will
be found. We were often interested in an initial segment of a given
length (for example, the first 100000 terms) and for that purpose we
used a reasonable guess for the expected density of the sequence, and
hence for the bound up to which we would like to compute, by quickly
generating a short initial segment (say up to bound a 1000 or so).

Since here the same set of combinations is visited as before, there is no immediate
shortcut. But it is possible to cache partial sums and update 
them whenever a new element of the sequence is determined. 
This prevents multiple recalculations of the partial sums,
and may offer a substantial gain. We used this (storing sums of pairs)
when computing large initial segments of $3$-sumfree sequences.
\subsection{Determining periods}\label{subsec:periods}
To prevent computing unnecessary large initial segments, it is good to use
Lemma \ref{lemma} to check if period and preperiod can be obtained from
a given initial segment. We explain how we used the lemma in the case of 
$t=3$.

Here is the general idea. We attempt to use Corollary \ref{corollary}
and the correspondence between characteristic and difference sequences
related to a 3-sumfree sequence, as outlined in the Introduction.
Suppose we have computed the first $n$ terms, where for simplicity of
exposition we assume that $n$ is a nine-fold $n=9h$. Now check if
blocks 2 to 9 are identical (and contain at least one 1): for $u=2$ to 
9 compare the blocks $d_{(u-1)h+1}, d_{(u-1)h+2}, \ldots, d_{uh}$
of the difference sequence $D=D_{f,g,h}$ of $S=S_{f,g,h}$. If
the blocks are all the same, we know that we have found (a multiple of) the
the period of $D$ (and hence of $S$), and that the preperiod is entirely
contained in the first block. 
The next step is to see where the first period of length $h$ starts,
that is, which forward shift still works,
by checking if $d_{h}=d_{2h}$, $d_{h-1}=d_{2h-1}$, $\ldots$ until the
first $j$ is found where $d_{h-j}\neq d_{2h-j}$. Finally, one determines
the proper period length by considering blocks starting at
$d_{h-j+1}$ of length the divisors of $h$, since we now know that 
blocks of length $h$ do repeat infinitely often.

In principle one could perform these checks whenever $n$, the length
of the computed initial segment, reaches a new multiple of 9. In practice,
essentially because of the expensiveness of the computing more terms and
the efficiency in caching partial sums, we computed large initial segments
before executing a period check. In that case it is useful not just to
check $n/9$ as a possible period, but also smaller values: take 8 blocks of
the same length $\ell$ from the end of the computed segment (containing at least one
1) and check if they are identical. If not, increment $\ell$ and repeat
until repetition is established or the maximum size $\ell=n/9$ is reached.
It may also be useful to check potential period lengths {\it larger}
than $n/9$, but if a candidate period is found one will have to compute
more terms to {\it prove} periodicity.

\section{2-sumfree sequences}\label{sec:t=2}
It will be clear that $t$-sumfree sequences are uninteresting for $t=1$: under our definition only the sequences $n, n+1, n+2, n+3,\ldots$, for any $n\geq 1$ qualify.

The case $t=2$ is more interesting, and shows phenomena also appearing for $t=3$.
It seems logical to consider first the sequences $S_{1,g}$ where the parameter $g$ is $\geq 2$; we will see though, that this may not be the most convenient parameterization.
In any case, $g$ even or odd induces different behaviour: if $g$ is odd $S_{1,g}$ only contains odd integers.
\begin{theorem}\label{thm:k2odd}
For any odd $g$ the greedy $2$-sumfree sequence $S_{1,g}$ is given by
$$z\in S_{1,g}\iff z=1\textrm{\ \ or\ \ }z\geq g \textrm{\ and odd}.$$
Thus, $S_{1,2e+1}$ is ultimately periodic modulo 2, with preperiod length 1 and period length 1, except for $S_{1,3}$ which is periodic from the beginning.
\end{theorem}

\begin{proof}
Let $g=2e+1$. Clearly $g+1$ is the sum of two
entries of $S_{1,g}$, hence not in it, while $g+2$ is not representable,
and thus forms the next entry. But then $g+3=1+(g+2)$ is representable again, and
$g+4$ is not. Etcetera: $S_{1,g}=1,g,g+2, g+4,g+6, \ldots$.
\end{proof}
The situation for even $g$ is slightly more complicated, and is most easily 
described using the sequence of differences. 
\begin{theorem}\label{thm:2sum1g}
Let $g=2e$ be even. For $g\geq 8$ the greedy 2-sumfree sequence $S_{1, g}$ is
characterized as follows:
\begin{eqnarray*}
z\in S_{1,g}&&\iff z\in\{1, g,g+2,g+4,\ldots,2g,2g+3\}\textrm{\ \ or\ \ }\\
&&z>2g+4 \textrm{\ and\ }\\
&& z\equiv\{2g+5, 2g+7, 2g+9,\ldots, 3g+1, 4g, 4g+2, 4g+4,\\
&&\phantom{z\equiv}4g+6, 5g+3, 5g+5, 5g+7,\ldots, 6g+1\}\bmod 4g+3.\end{eqnarray*}
In particular, for $g\geq 8$ after the first $e+3$ entries (in a preperiod) the sequence $S_{1,g}$ modulo $4g+3$ is periodic with period length $g+3$.
\end{theorem}
\begin{example}\rm
Consider the case $g=12$:
\begin{eqnarray*}
S_{1,12}&=&(1, 12, 14, 16, 18, 20, 22, 24, 27, 29, 31, 33, 35, 37, 48, 50, 52, \\
&&54, 63, 65, 67, 69, 71, 73, 80, 82, 84, 86, 88, 99, 101, 103, 105, \\
&&114, 116, 118, 120, 122, 124, \ldots).
\end{eqnarray*}
The preperiod is $1,12,14,16,18,20,22,24,27$ and after that we find a block
$$29, 31, 33, 35, 37, 48, 50, 52, 54, 63, 65, 67, 69, 71, 73$$ of length
15, with differences
$$2,2,2,2,11,2,2,2,9,2,2,2,2,2,7$$
where the last difference takes us to the beginning of the next block
$$80, 82, 84, 86, 88, 99, 101, 103, 105, 114, 116, 118, 120, 122, 
124$$
with the same differences. It is the repetitiveness of this difference
sequence that forms the true periodicity. Note that the modulus of repetition 
is simply the sum (51) of the 15 differences in each block.
\end{example}
\begin{proof}[Proof of Theorem \ref{thm:2sum1g}]
We sketch the proof, leaving some details to the reader to check.

Let $g\geq 8$ be even, denote $S=S_{1,g}$ and let $m=4g+3$.
Consider blocks $B_j=[1+(j-1)\cdot m, \ldots, j\cdot m\}$ of $m$ integers,
for $j\geq 1$; we wish to prove that
$$S\cap B_1=\{1\}\cup\{g, g+2, \ldots, 2g\}\cup\{2g+3, \ldots, 3g+1\}\cup
\{4g, 4g+2\}$$
and $S\cap B_2=$
$$\{4g+4, 4g+6\}\cup\{5g+3, \ldots, 6g+1\}\cup\{6g+8, \ldots, 7g+4\}\cup
\{8g+3, 8g+5\}$$
and generally, that  for $j\geq 2$ it holds that $S\cap B_j=$ is obtained
by adding $(j-1)\cdot m$ to each of
$$\{1, 3\}\cup\{g, \ldots, 2g-2\}\cup\{2g+5, \ldots, 3g+1\}\cup
\{4g, 4g+2\};$$
the $\ldots$ here, and in all cases in this proof, indicate steps of $2$.

Since $1, g$ are the first elements of $S$ by definition, it is easy 
to see that alternately the even numbers $g, g+2, \ldots, 2g$ are in $S$ but
the odd numbers $g+1, g+3, \ldots, 2g+1$ are not.
Then the even numbers $2g+2=g+(g+2), 2g+4=g+(g+4), \ldots, 3g=9+2g$ have
to be omitted, but the odd $2g+3, 2g+5, \ldots, 3g+1$ will be included
in $S\cap B_1$.
From here on several even and odd numbers are 2-sums: 
$$3g+2=(g+2)+2g, 3g+4=(g+4)+2g, \ldots, 4g-2=(2g-2)+2g$$
and
$$3g+3=g+(2g+3), 3g+5=g+(2g+5), \ldots, 4g+1=g+(3g+1).$$ 
At the end of the first block, neither $4g$ nor $4g+2$ are sums of two
previous entries of $S$, but $4g+3=(g+2)+(3g+1)$
and will be omitted. Apart from the careful verification that
the remaining integers will indeed be in $S\cap B_1$,
this proves the first part.

For the second block we will give several subsets that have to
be excluded from $S$, and with arguments that can be
generalized for all following blocks. To check that the remaining
numbers will be in $S$ needs verification that they are not 2-sums;
this is easy and a finite task, but tedious and hard to write down
convincingly; for later blocks we will add some justification.

It will be useful to note first that neither $4g+4$ nor $4g+6$ are
2-sums from the first block; as we will see that $4g+5$ is not in $S$,
that suffices to see they will be the first pair of entries in $S\cap B_2$.

Now use that the odd numbers $2g+5, \ldots, 3g+1$ are in $S\cap B_1$ as well
as $2g$ and $2g+3$ to see that
the following are not in $S\cap B_2$: 
$$X_1=\{4g+5=2g+(2g+5), \ldots, 5g+1=2g+(3g+1)\}$$
and
$$X_2=\{4g+8=(2g+3)+(2g+5), \ldots, 5g+4=(2g+3)+(3g+1)\}.$$
Since also all even numbers $g, \ldots, 2g-2$ are in $S\cap B_1$
while we already found that $4g$ and $4g+2$ are in $S\cap B_1$
and $4g+4, 4g+6$ are in $S\cap B_2$,
we get that
$$X_3=\{5g+4=g+(4g+4), \ldots, 6g+2=(2g-2)+(4g+4)\}$$
and
$$X_4=\{6g=2g+4g, \ldots, 6g+6=2g+(4g+6)\}$$
are to be excluded from $S$ as well.
At this point we need to verify that $g-1$ odd elements remain in $S$
$$S\cap B_2\supset\{5g+3, \ldots, 6g+1\}$$
since they are not 2-sums of elements in $S$ so far. We can then use these
to exclude the odd numbers
\begin{align*}
X_5=&\{6g+3=(5g+3)+g, \ldots, 7g+1=(6g+1)+g\}\cup\\
&\cup\{7g+3=(g+2)+(6g+1), \ldots, 8g+1=2g+(6g+1)\}
\end{align*}
and the even numbers
\begin{align*}
X_6=&\{7g+6=(2g+3)+(5g+3), \ldots, 8g+4=(2g+3)+(6g+1)\}\\%\cup\\
&\cup\{8g+6=(2g+5)+(6g+1)\}.
\end{align*}
We will leave to the reader the final verification for this block, namely
$$S\cap B_2\supset\{6g+8, \ldots, 7g+4\}\cup\{8g+3, 8g+5\},$$
as they cannot be obtained as sums of distinct previous elements.

For block $j$, with $j\geq 3$, it is easy to see by induction
that shifts of the subsets $X_i$, for $1\leq i\leq 6$ will have
to be omitted from $S\cap B_j$: in $X_1$ and $X_2$ adding $(j-2)\cdot m$
to the second summands gives elements of $S\cap B_{j-1}$, and
shows that these sets have to be excluded from $S\cap B_j$.
Likewise, adding $(j-2)\cdot m$ to the first summands in $X_3$
yields elements from $S\cap B_{j-1}$ and shows that a shift of 
$X_3$ by $(j-2)\cdot m$ has to be excluded from $S\cap B_j$.
For $X_4$ a shift of the second argument can be seen to land
in $S\cap B_{j-2}$ or $S\cap B_{j-1}$ and leads to an excluded set of
4 elements from $S\cap B_j$ again. For $X_5$ and $X_5$
we first verify that adding $(j-1)\cdot m$ to their first arguments
(they concide) lands in $S\cap B_j$, and adding elements from $S\cap B_1$
yields shifted versions of $X_5$ and $X_6$ to be excluded from $S\cap B_j$.

For the verification that the elements from block $B_j$ that
were not yet excluded, do indeed remain in $S$, the following
inductive argument can be made precise: once we have found
all elements of $S\cap B_i$ for $i\leq j-1$, with the only exception
of the element $4g+6$ from $S\cap B_2$, every element $s$ from
$\cup_{i=2}^{j-1} (S\cap B_i)$ %with $2\leq j-1$ 
has the property that $s-m$ is in $\cup_{i=2}^{j-1} (S\cap B_{i-1})$.
This implies that for $s\in B_j$ any 2-sum $s=u+v$ with distinct smaller
elements $u, v\in S$ leads to a 2-sum for $s-m$ by subtracting $m$ from
one of $u, v$: so all elements that survived in $S\cap B_{j-1}$ will also
survive in $S\cap B_j$, unless one of the following exceptional
situations occurs:
both $u, v$ are in $S\cap B_1$, or $u\in S\cap B_1$ and $v=4g+6$,
or $u\in S\cap B_1$ and $v=u+m$. For any of these 2-sums it is \emph{not}
possible to subtract $m$ from either summand. Thus, there are only
finitely many cases left to consider: most of these 2-sums end up
in $B_2$ and had to be removed from $S\cap B_2$ already, and the remaining 
exceptional 2-sums can simply be shown to be non-exceptional 2-sums
with different summands as well, and hence are no obstruction
to the induction argument.
\end{proof}
\begin{remarks}\rm

The Theorems \ref{thm:k2odd}, \ref{thm:2sum1g}, \ref{thm:2subs} describe infinite families of `similar' 
2-sumfree sequences, in 
a parameterized fashion that will recur for families with $t=3$ further on as
well. Typically, there are finitely many `exceptions', cases not covered, with
deviating (sometimes larger than expected) moduli and periods. 
For the family in Theorem \ref{thm:2sum1g} the exceptional cases $g=2, 4, 6$ give rather easy sequences:
$$S_{1,2}=(1,2,4,7,10,13,16,19,22,25,\ldots),$$
which contains $2$ and the single residue class $1\bmod 3$;
$$S_{1,4}=(1, 4, 6, 8, 11, 13, 16, 18, 23, 25, 28, 30, 35, 37, 40, \ldots)$$
with $8$ and the residue classes of $1, 4, 6, 11$ modulo $12$; and
$$S_{1,6}=1, 6, 8, 10, 12, 15, 17, 19, 24, 26, 28, 33, 35, 37, 42, 44, 46, 51, 53,\ldots$$
containing $12$ and the
complete residue classes $1, 6, 8\bmod 9$.

Also, our proofs are very elementary, but lengthy and tedious to check. Their
structure (also for $t=3$) is usually similar to that given above: omit and
include the integers in the first blocks and then use induction, with some
care about exceptions, for the next blocks.
\end{remarks}
The next theorem describes another infinite family, which we mainly give for comparison with the
case $t=3$: it turns out that families of this kind, with fixed difference
between two parameters, are easier to describe explicitly.
\begin{theorem}\label{thm:2subs}
Let $d\geq 1$. For every $f\geq d+1$ the 2-sumfree sequence $S_{f,f+d}$ is
characterized as follows:
\begin{eqnarray*}
z\in S_{f,f+d}&\iff& z\in\{f,2f+d-1\}\textrm{\ \ or\ \ }\\
&&\hskip-1truecm z>f+d-1 \textrm{\ and\ }\\
&&\hskip-1truecm z\equiv\{f+d-1, f+d,\ldots, 2f+d-2\}\bmod 3f+d-1.
\end{eqnarray*}
In particular, after the first $f+1$ entries (in a preperiod) the sequence $S_{f,f+d}$ modulo $3f+d-1$ is periodic with period length $f$.
\end{theorem}

\begin{proof}
Let $d\geq 1$, $f\geq d+1$, and write $m=3f+d-1$ and $S=S_{f,f+d}$. 
We consider the blocks $B_j=\{1+(j-1)\cdot m, \ldots, j\cdot m\}$ and
sketch the proof that 
$$S\cap B_1=\{f, f+d, f+d+1, \ldots, 2f+d-1\}$$
while for $j\geq 2$
$$S\cap B_j=\{f+d-1+j\cdot (m-1), \ldots, 2f+d-2+j\cdot(m-1)\}.$$
Then $S$ consists of the $f$ residue classes $f+d-1, \ldots, 2f+2d-2$
modulo $m$, except that $f+d-1$ is missing but $f$ and $2f+d-1$ are included.

The first block is easy: the first entries of $S$ are $f$
and $f+d$, and after that the first 2-sum is $2f+d$, so
$f+d-1, f+d-2, \ldots, 2f+d-1$ are in $S\cap B_1$. But then
$2f+d=f+(f+d), \ldots, 3f+d-1=f+(2f+d-1)$ are excluded as 2-sums.

In the second block the initial segment $3f+d, \ldots, 4f+d-1$ must
be omitted from $S$ as they are all 2-sums from the first block:
\begin{align*}
3f+d&=(f+d+1)+(2f-1),\\
&\ \ \vdots\phantom{(f+d+1)+(2f-}\vdots\\
3f+2d&=(f+d+1)+(2f+d-1),\\
%&\vdots\\
3f+2d+1&=(f+d+2)+(2f+d-1),\\
&\ \ \vdots\phantom{(f+d}\vdots\\
4f+2d-3&=(2f+d-2)+(2f+d-1).
\end{align*}
This exhausts all 2-sums from $S\cap B_1$, and shows that
$4f+2d-2, \ldots, 5f+2d-3$ is entirely contained in $S\cap B_2$.
The next 2-sums will be obtained by adding $f$ to these:
$5f+2d-2=f+(4f+2d-2), \ldots, 6f+2d-3=f+(5f+2d-3)$.
Finally, the last segment of $B_2$ is excluded from $S$, as is
the initial segment of $S\cap B_3$, since
\begin{align*}
6f+2d-3&=(2f-1)+(4f+2d-2),\\
&\ \ \vdots\quad\quad\vdots\\
6f+3d-3&=(2f+d-1)+(4f+2d-2),\\
6f+3d-2&=(2f+d-1)+(4f+2d-1),\\
&\ \ \vdots\phantom{(2f+d-1)+(4f+2}\vdots\\
7f+3d-4&=(2f+d-1)+(5f+2d-3).
\end{align*}
Once we have completely determined the first few blocks,
it is not difficult to proceed by induction to obtain
the next block by induction: adding $m$ to one summand
determines exclusions and if $x$ in block $j$ is
a 2-sum from previous entries, the same holds for 
$x-m$ in the previous block. Thus the exclusions and
inclusions in block $j$ match up precisely with those
in block $j-1$, for $j\geq 3$.

Finally, one immediately observes that a proper period
starts after the $f+1$ entries $S\cap B_1$ in the first block.
\end{proof}

\begin{example}\rm
For $f=4$ and $d=3$ the above theorem describes the sequence
$$S_{4,7}=(4,7,8, 9, 10, 20, 21, 22, 23, 34, 35, 36, 37, 48, 49, 50, \ldots)$$
as consisting of $4, 10$ and the residue classes of $6, 7, 8, 9$ modulo 14,
except that 6 itself is missing.
\end{example}

\begin{remarks}\rm
Theorem \ref{thm:2subs} solves, for fixed $d$, the periodicity problem
for all but finitely many $f$: only the cases $S_{f, f+d}$ for $1\leq f\leq d$ are not covered. In fact it is not very difficult to state the result
for $\frac{2}{3}$ of this remaining interval, although it looks slightly messy (even without being precise about the preperiod), and giving a clean uniform proof is hard. For that reason we have not included the details
here.
In the remaining cases, for $f\leq q=d/3$ the behaviour of
the period length is more erratic, but it is usually comparable
to $d$ in this range. This is not true for the length of the preperiod:
this may be in the order of a few thousands even for small $f, d$: the
sequence $S_{3, 98}$ has a preperiod of 2566 terms and period length 2.

For details on all of the remaining cases for $t=2$, please consult \cite{BB2}.
\end{remarks}
\section{3-sumfree sequences}\label{sec:t=3}
In this section, the main one of this paper, we consider 3-sumfree sequences.
As noted, they are specified by the first 3 values, 
denoted by $f<g<h$ for $S_{f,g,h}$, but
we will find it convenient to consider families of these sequences with fixed
difference $d$ between $g$ and $h$. So we look at $S_{f,g,g+d}$.

Our first result considers the family with $f=1$ and $d=1$, and includes the sequence
$S_{1,2,3}$, which inspired this work. Again, we emphasize that our sumfree sequences are all greedy and strict.
\begin{theorem}\label{thm:k=1d=1}
For every $g\geq 2$ the 3-sumfree sequence $S_{1,g,g+1}$ is 
characterized as follows:
\begin{eqnarray*}
z\in S_{1,g,g+1}&&\iff\quad
z\in\{1, 2g + 1, 6g + 1\}\textrm{\ \ or\ \ }\\
&&\hskip-2truecm z\equiv\{g, g + 1, \ldots, 2g\}\cup\{6g + 2, 6g + 3, \ldots, 7g + 1\}\bmod 10g+3.
\end{eqnarray*}
In particular, after the first $g + 4$ entries, the sequence modulo $10g + 3$ is periodic
with period $2g + 1$.
\end{theorem}
\begin{proof}
We consider blocks of $10g+3$ integers, and examine which numbers
in the block will belong to $S=S_{1,g,g+1}$. It turns out that the
first block is different, but from the second block on a pattern
repeats.

We abbreviate $m=10g+3$ throughout this proof, and
write $B_j=[1+(j-1)\cdot m, \ldots, j\cdot m]$ for the $j$-th block
of length $m$. Then we
prove 3 claims, where the second and third perform induction 
on blocks:
\begin{itemize}
\item[(i)] For the first block $B_1=[1,\ldots, m]$ holds
$$S\cap B_1=\{1\}\cup\{g,\ldots,2g+1\}\cup\{6g+1,\ldots,7g+1\};$$
\item[(ii)] For the second block: 
$$S\cap B_2=\{11g+3\ldots,12g+3\}\cup\{16g+5,\ldots,17g+4\};$$
\item[(iii)] Block $j+1$, that is $B_{j+1}=[1+j\cdot m,\ldots, (j+1)\cdot m]$,
for $j\geq 2$ satisfies $S\cap B_{j+1}=$
\begin{equation*}\label{eq:iii}
\{g+j\cdot m\ldots,2g+j\cdot m\}\cup\{6g+2+j\cdot  m,\ldots,7g+1+j\cdot  m\}.
\end{equation*}
\end{itemize}

\begin{proof}[Proof of (i)]
The first entries are $1, g, g+1$ by definition; their sum $2g+2$ will be the
first non-entry after $g-1$ so the block contains $\{g, \ldots, 2g+1\}$.
Then all of $\{2g+2, \ldots, 6g\}=$
\begin{align*}
&\{1+g+(g+1), \ldots, 1+g+(2g+1),\ldots,1+2g+(2g+1),\\
& g+(g+2)+(2g+1), (g+1)+(g+2)+(2g+1),\ldots\\
& \ldots, (g+1)+2g+(2g+1), \ldots, (2g-1)+2g+(2g+1)\}
\end{align*}
is excluded.

The next entry is $6g+1$ and the next smallest sum of 3 entries
of $S$ becomes $1+g+(6g+1)=7g+2$, so
$$\{6g+1, \ldots, 7g+1\}\subset S,$$
and all of $\{7g+2,\ldots, 10g+3\}=$
\begin{align*}
&\{1+g+(6g+1), \ldots, 1+(2g+1)+(6g+1),\\
& (g+1)+(g+2)+(6g+1), \ldots, (g+1)+2g+(6g+1),\ldots\\
& \ldots, (2g-1)+2g+(6g+1), (2g-1)+(2g+1)+(2g+1),\\
& 2g+(2g+1)+(6g+1), 2g+(2g+1)+(6g+2)\}
\end{align*}
is excluded again.
\end{proof}
\begin{proof}[Proof of (ii)]
Continuing our `exclusion' and `inclusion' process we find for the second
block that all of 
\begin{align*}
\{10g+4, &\ldots, 11g+2\} =\\
&\{(g+2)+(2g+1)+(7g+1), \ldots, 2g+(2g+1)+(7g+1)\}
\end{align*}
is excluded but $11g+3$ is not (as is easily verified). Since 
the next smallest 3-sum then is 
$1+g+(11g+3)=12g+4$, we have
$$\{11g+3, \ldots, 12g+3\}\subset S,$$
We have to exclude all of $\{12g+4, \ldots, 16g+4\}=$
\begin{align*}
&\{1+g+(11g+3), \ldots, 1+g+(12g+2),\ldots,\\
&\ldots, 1+2g+(12g+3), g+(g+2)+(12g+3),\\
& (g+1)+(g+2)+(2g+3),\ldots\\
& \ldots, (g+1)+2g+(12g+3), \ldots, (2g-1)+2g+(12g+3,\\
&(2g-1)+2g+1+(12g+3), 2g+(2g+1)+(12g+3)\},
\end{align*}
but $16g+5$ cannot be written as a sum of entries of $S$.
The smallest 3-sum then will be %not be $1+g+(16g+5)=17g+6$, but
$1+(11g+3)+6g+1=17g+5$ so
$$\{16g+5, \ldots, 17g+4\}\subset S.$$
Also to be excluded is the remainder of this second
block,  namely $\{17g+5, \ldots, 20g+6\}=$
\begin{align*}
&\{1+(11g+3)+(6g+1), 1+g+(16g+5),\ldots, \\
& 1+(2g+1)+(16g+5), \ldots, (g+1)+(g+2)+(16g+5), \ldots\\
& \ldots, (g+1)+2g+(16g+5),\ldots, (2g-1)+2g+(16g+5),\\
& (2g-1)+(2g+1)+(16g+5), 2g+(2g+1)+(16g+5)\}.
\end{align*}
\end{proof}
\begin{proof}[Proof of (iii)]
Finally, assuming (i) for $B_1$, and that 
\begin{align*}S\cap B_k&=
\{g+(k-1)\cdot m\ldots,2g+(k-1)\cdot m\}\cup\\
&\cup\{6g+2+(k-1)\cdot  m,\ldots,7g+1+(k-1)\cdot  m\}
\end{align*}
for  $2\leq k\leq j$, we obtain the corresponding result for $S\cap B_{j+1}$
in the same fashion as we did for $k=2$: whenever we used an element
$b+m$ (with $b\in S\cap B_1$), from $S\cap B_2$ in a 3-sum to exclude 
the sum from $S$, we now replace this element by $b+j\cdot m$.
This proves that if $z\in B_2$ is not in $S$, then
also $z+(j-1)\cdot m\notin S$.

Remains to show that if $z\in S\cap B_2$,
also $z+(j-1)\cdot m\in S\cap B_{j+1}$. That is, no further
elements from $B_{j+1}$ are eliminated from $S$. The idea will be
to show that if $z\in B_{j+1}$ can be written as 3-sum $z=a+b+c$
with distinct $a, b, c\in S$, then so can $z-m$, essentially because we can
subtract $m$ from one of the summands, say $x$. This will work,
unless either $x-m\notin S$ or $x-m$ equals one of the other two summands.
This, in turn, will only be the case for each $x$ if one the following holds:
\begin{itemize}
\item[(i)] $a,b,c\in S\cap B_1$;
\item[(ii)] $a, b\in S\cap B_1$ and $c=a+m$ or $c=b+m$;
\item[(iii)] $a\in S\cap B_1$ and $b=a+m$ and $c=a+2m$;
\end{itemize}
This leaves only a few small cases to consider, and it turns out
that none of these cases causes a problem: either
their sum $a+b+c$ will be excluded from $B_2$ and
$B_3$ already (as is easily verified if $a,b,c\in B_1$),
or it is possible to write $z=a+b+c$ 
also as $z=a'+b'+c'$ with $a',b',c'\in S$ not
satisfying any of (i)--(iii). The latter can usually be achieved
by incrementing one of the summands (by 1) and decrementing another
(here it usuful to know that the subintervals in $S\cap B_1$ contain
at least $g+1\geq 3$ elements); the only situation where this does not
work is when $a=g$ or $a=7g+1$ and $b=a+m, c=a+2m$.
In both cases the sum ($33g+9$ or $51g+12$) lies in the already excluded
middle section (of $B_4$ or $B_6$).

We conclude that if $z$ is a 3-sum and
must be excluded from $B_{j+1}$, then $z-m$
had to be excluded from $B_j$.
\end{proof}
By looking at $S\cap B_1$ and $S\cap B_2$ we see immediately
that periodicity starts after the 3 entries from the first block
and the first $g+1$ entries from the second block.

This concludes the proof of the theorem.
\end{proof}

An independent proof by Orion Shtrezi was
published before in \cite{shtrezi}.

\begin{example}\label{ex:orig}\rm
The original sequence ($g=2$), according to this theorem, is then 
$$S_{1,2,3}=(1,2,3,4,5,13, 14, 15, 25, 26, 27, 37, 38, 48, 49, 50, 60, 61, 71, % 72, 73, 83, 84, 94, 95, 96
\ldots)$$
and will contain, besides $1,5,13$, all integers
that equal $2, 3, 4$ or $14, 15$ modulo 23.
\end{example}
The next theorem concerns a 2-parameter family for $f=1$; note that the previous 
theorem is {\it not} obtained by taking $d=1$ here.
\begin{theorem}\label{thm:f=1d>1}
Let $d\geq 2$. For every $g\geq d+1$ the greedy 3-sumfree sequence $S_{1,g,g+d}$ is
characterized as follows:
\begin{eqnarray*}
z\in S_{1,g,g+d}&\iff& z\in\{1,g,2g+d-1,2g+d\}\textrm{\ \ or\ \ }z\geq g+d \textrm{\ and\ }\\
&&\hskip-2truecm z\equiv\{g+d-2, g+d-1,\ldots, 2g+d-2\}\bmod 5g+2d.\end{eqnarray*}
In particular, for $d\geq 2$ and every $g\geq d+1$ after the first $g+3$ entries (in a preperiod) the sequence $S_{1,g,g+d}$ modulo $5g+2d$ is periodic with period $g+1$.
\end{theorem}
\begin{proof}
Our proof is similar to that of the previous theorem but slightly
easier, as there will be only one surviving segment in each block,
rather than 2. We write $S=S_{1,g,g+d}$ and put $m=5g+2$,
and let the block $B_j$, for $j\geq 1$, consist of integers 
$\{1+(j-1)\cdot m,\ldots, j\cdot m\}$.
We show that
$$S \cap B_1=\{1\}\cup\{g\}\cup\{g+d,\ldots, 2g+d\},$$
and by induction that for $j\geq 1$:
$$S\cap B_{j+1}=\{g+d-2+j\cdot m, \ldots, 2g+d-2+j\cdot m\}.$$

By definition $1, g, g+d$ are in $S$ and the first non-entry is
their sum $2g+d+1$; thus
$$S\cap B_1\supset\{1\}\cup\{g\}\cup\{g+d, \ldots, 2g+d\}.$$
But the remainder of the first block is excluded from $S$ as can be
seen from: $\{2g+d+1, \ldots, 5g+2d\}=$
\begin{align*}
\{ &1+g+(g+d), \ldots, 1+g+(2g+d), 1+(g+d)+(g+d+1), \ldots\\
&\ldots, 1+(g+d)+(2g+d), \ldots, 1+(2g+d-1)+(2g+d),\\
& g+(g+d+1)+(2g+d), \ldots, g+(2g+d-1)+(2g+d),\\
& (g+d)+(2g+1)+(2g+d-1)\}.
\end{align*}
So 
$$S\cap B_1=\{1\}\cup\{g\}\cup\{g+d, \ldots, 2g+d\}.$$

The list of omissions from $S$ continues in the second block:
\begin{align*}
\{&5g+2d+1, \ldots, 6g+3d-3\}=\\
&\{(g+d)+(2g+1)+(2g+d), \ldots, (g+d)+(2g+d-1)+(2g+d),\\% \ldots\\
&\ldots, (2g+d-2)+(2g+d-1)+(2g+d)\}.
\end{align*}
This is the largest 3-sum from the entries so far, so
$6g+3d-2$ is not a sum of 3 previous entries, and forms the next
element of $S$. Then the next omission will be $7g+3d-1=1+g+(6g+3d-2)$,
and hence
$$S\cap B_2\supset\{6g+3d-2, \ldots, 7g+3d-2\}.$$
The remainder of the second block is excluded again:
\begin{align*}
\{&7g+3d-1, \ldots, 10g+4d\}=\\
&\{1+g+(6g+3d-2), \ldots, 1+g+(7g+3d-2),\\
& 1+(g+d)+(7g+2d-1), \ldots, 1+(g+d)+(7g+3d-2), \ldots\\
& \ldots, 1+(2g+d)+(7g+3d-2), g+(g+d+2)+(7g+3d-2), \ldots\\
& \ldots, g+(2g+d)+(7g+3d-2), (g+d)+(2g+1)+(7g+3d-2), \\
& (g+d)+(2g+2)+(7g+3d-2)\},
\end{align*}
and therefore
$$S\cap B_2=\{6g+3d-2, \ldots, 7g+3d-2\}.$$
For the induction step we know that
$$S\cap B_1=\{1\}\cup\{g\}\cup\{g+d, \ldots, 2g+d\}.$$
and we assume that for $i=2, \ldots, j$ we have already proven
$$S\cap B_j=\{g+d-2+(j-1)\cdot m, \ldots, 2g+d-2+(j-1)\cdot m\};$$
It is easy to show that then the first and the last part of block
$B_{j+1}$ wil be excluded from $S$ again, just as we did for $B_2$,
by simply replacing $m$ by $(j-1)\cdot m$ in the first summand:
\begin{align*}
\{&1+j\cdot m, \ldots, g+d-3+j\cdot m\}=\\
&\{(g+d+(j-1)\cdot m)+(2g+1)+(2g+d), \ldots\\
&(g+d+(j-1)\cdot m)+(2g+d-1)+(2g+d), \ldots\\
&\ldots, (2g+d-2+(j-1)\cdot m)+(2g+d-1)+(2g+d)\}.
\end{align*}
or by replacing $7g+3d-2$ for the second block by
the largest element $L=2g+d-2+(j-1)\cdot m$ of $S\cap B_j$:
\begin{align*}
\{&2g+d-1+j\cdot m, \ldots, 5g+2d+j\cdot m\}=\\
&\{1+g+(g+d-2+(j-1)\cdot m), \ldots, 1+g+L,\\
& 1+(g+d)+(2g-1+(j-1)\cdot m), \ldots, 1+(g+d)+L, \ldots\\
& \ldots, 1+(2g+d)+L, g+(g+d+2)+L, \ldots\\
& \ldots, g+(2g+d)+L, (g+d)+(2g+1)+L, (g+d)+(2g+2)+L\}
\end{align*}
where one easily checks (using that $d\leq g-1$) that in each triple
the summands are different and either in $S\cap B_1$ or $S\cap B_j$.

Remains to show that every $z\in\{g+d-2+j\cdot m, \ldots, 2g+d-2+j\cdot m\}$
is contained in $S$. We intend to show that if one such $z$
is the sum of 3 previous entries of $S$, this implies that
we obtain $z-m\in S\cap B_j$ as a 3-sum and hence a contradiction. 
We obtain $z-m$ by subtracting $m$ from one of the summands $w$ for $z$;
such subtraction is allowed, except when $w$ is in the first block,
or $w-m$ equals one of the other two summands, or $w-m\geq 1$ is not
in $S$, which can happen only (since $w\in S$) in the second block,
when $w=6g+3d-2$ or $w=6g+3d-1$.

Thus we do immediately derive a contradiction from a 3-sum $z=a+b+c$
unless one of the following holds:
\begin{itemize}
\item[(i)] $a,b,c\in S\cap B_1$;
\item[(ii)] $a,b\in S\cap B_1$ and $c\in\{6g+3d-2, 6g+3d-1\}$;
\item[(iii)] $a, b\in S\cap B_1$ and $c=a+m$ or $c=b+m$;
\item[(iv)] $a\in S\cap B_1$ and $b=a+m$ and $c=a+2m$;
\item[(v)] $a\in\{6g+3d-2, 6g+3d-1\}$ and $b=a+m$ and $c=a+2m$;
\end{itemize}
It is easy to see that if (i) holds, $a+b+c\in B_2$, so it cannot equal $z$.
If (ii) holds the sum $a+b+c$ lies in the gap for $S$ between $S\cap B_2$
and $S\cap B_3$, so is too small again. In the three remaining cases we
will show that the sum $z=a+b+c$ can be replaced by $z=a'+b'+c'$ for which
neither of these 5 cases applies, and hence $m$ can be subtracted from
a summand, which leads to a contradiction: the 3-sum cannot exist.
Here we will use that $d\geq 2$ and $g\geq d+1$, whence $S\cap B_i$
contains $g+1\geq 4>2$ integers.
Namely, if (iii) holds we can increment or decrement one of $a, b$ (without
clash and inside $S\cap B_1$) and do the opposite with $c$ to achieve that
$c-m$ is different from $a$ and $b$. (The only situation where this
is not possible is when $a=g+d$ and $b=g+d+1$, but then $a+b+c$ is in $B_2$).
If (iv) holds, we achieve the same by replacing $a$ by $a+1$ and $c$
by $c-1$. In case (v) we could do the same if $a=6g+3d-1$, but note
that we consider only 2 cases here, and for both the sum $a+b+c$
(equal to $33g+15d-6$ or $33g+15d-3$) is in the gap for $S$ in between
$\{31g+13d-2, \ldots 32g+13d-2\}=S\cap B_7$ and $S\cap B_8=\{36g+15d-2, \ldots, 37g+15d-2\}$
\end{proof}
In the following 2 theorems we specify the sequence $S_{f,g,h}$ by explicitly
giving its first difference sequence $D_{f,g,h}$, that is, $(S_{f,g,h}[i+1]-S_{f,g,h}[i])_{i=1}^\infty$. 
\begin{theorem}\label{thm:gend>1}
Let $d\geq 2$, and if $d=2$ assume $(f,g)\neq (1,2)$. Then for every $f\geq 1$ and every 
$g\geq f+1$ such that $f+g\geq d+1$
the greedy 3-sumfree sequence $S_{f,g,g+d}$ is ultimately periodic modulo $5g+2d+3(f-1)$, with preperiod length $g+f+2$ and period length $g+f$. More explicitly, $D_{f,g,g+d}$ is ultimately periodic in these cases, with preperiod
$$(g-f, d, 1^{f+g-1}, 4g+2f+2(d-2))$$
and period
$$(1^{f+g-1}, 4g+2f+2(d-1)).$$
\end{theorem}
\begin{proof}
We will sometimes use the notation $h=g+d$ for the third entry of the
sequence, and shorthand $S$ for $S_{f, g, h}$.

Step 1: the three initial entries $f, g, h=g+d$ in $S$
will be followed by a contiguous block of $f+g-1$ entries $g+d+1=h+1, \ldots, f+g+h-1$
because $f+g+h$ is the smallest sum possible. 

Step 2: after this there will be a contiguous block of integers
$f+g+h, \ldots, 3f+3g+3h-6$ {\it not} in $S$;
this will use the condition that $f+g\geq d+1$. If the condition is not satisfied there may be elements in this block that {\it do} belong to $S$.
It will be clear that $f+g+h$ is not in $S$, but by the previous step
 $h+1, h+2, \ldots, f+g+h-1$ are in $S$ like $h$ is, and therefore
$f+g+(h+1), \ldots, f+g+(f+g+h-1)=2(f+g)+h-1$ will not be in $S$. Likewise,
$f+h+(h+1), f+h+(h+2), \ldots, f+h+(f+g+h-1)=2(f+h)+g-1$ will not be in $S$.
This block will either overlap with, or follow directly after the previous
block (without holes) provided its first entry  is at most 1 more than the final entry in the first block: $f+2h+1\leq 2f+2g+h$, that is, $f+g\geq d+1$.

Now we can increment the second summand repeatedly, to obtain the
following blocks of elements not in $S$:

{\small
\begin{eqnarray*}
f+(h+1)+h+2, \ldots, f+(h+1)+f+g+h-1&=&2(f+h)+g,\\
f+(h+2)+h+3, \ldots, f+(h+2)+f+g+h-1&=&2(f+h)+g+1,\\
\ldots,
f+(f+g+h-2)+f+g+h-1&=&3f+2g+2h-3.
\end{eqnarray*}

}

This gives a contiguous block 
$$f+2h+3, \ldots, 2(f+g+h)+f-3$$
of integers not in $S$. This block clearly overlaps with the previous.

Next we look for 3-sums in $S$ with $g$ instead of $f$ as a first summand;
this will overlap with the previous since the least element
$g+h+(h+1)=g+2h+1$ is clearly smaller than $3f+2g+2h-3$.
We then find these elements can not be in $S$:

	{\small
\begin{eqnarray*}
g+h+(h+1), \ldots, g+h+f+g+h-1=2(g+h)+f-1,\\
g+(h+1)+(h+2), \ldots, g+(h+1)+f+g+h-1=2(g+h)+f,\\
\ldots,
g+(f+g+h-2)+(f+g+h-1)=2f+3g+2h-3.
\end{eqnarray*}
	}

This extends the block of forbidden entries to $2f+3g+2h-3$.

Finally, we further extend this by replacing $g$ by $h$ as a first
summand, which only leaves a hole with the existing block if
$$h+(h+1)+(h+2)>(2f+3g+2h-3)+1,$$
that is, when $d=h-g>2(f+g)-5,$
which for $f+g\geq d+1$ and $d\geq 2$ only holds when $(f,g,h)=(1,2,4)$,
a case excepted in this theorem.
We find that these blocks will also not be in $S$:
	{\small
\begin{eqnarray*}
h+(h+1)+h+2, \ldots, h+(h+1)+f+g+h-1=f+g+3h,\\
h+(h+2)+h+3, \ldots, h+(h+2)+f+g+h-1=f+g+3h+1,\\
\ldots,
(f+g+h-3)+(f+g+h-2)+(f+g+h-1).
\end{eqnarray*}
}

Altogether we obtain that the block of integers
$$f+g+h, f+g+h+1, \ldots, 3(f+g+h)-6$$
is not in $S$.

Step 3: in the next step we obtain a block of integers in $S$.
The sum plus 1 of the largest 3 elements found so far will be the next in $S$.
This is $M=3(f+g+h-1)-2$. Its successors, up to $M+f+g-1$ form a block
in $S$ of $f+g$ consecutive entries in $S$. This block consists of
the block of entries found in Step 1 shifted by $3(f+g)+2h-5$.

Step 4: Now we repeat the procedure described in Step 2 to find a block
of consecutive integers of length $2(f+g+h)-5$, except that we now use the
block of $f+g$ integers from Step 3 instead of the $f+g$ integers from Step 1.

Step 5: Repeat the procedure from Step 3 to obtain a consecutive block
of non-entries for $S$ using the block of elements from Step 4 instead
of the initial block of $f+g$ entries.

Then repeat Step 4 and Step 5 ad infinitum.

There is one important difference between the first block of
omissions and the subsequent ones. The $k$th block of omissions starts
with the entry $f+g+B^{k}_1$ and end with $B^{1}_{f+g-1}+B^{1}_{f+g}+B^{k}_{f+g}$
except when $k=1$: then $B^{k}_{f+g}=B_1{f+g}$ which is used in the sum
already and it must be replaced by $B_{f+g-1}$. This explains why
the first block of omissions has size $2(f+g+h)-3$ while all subsequent
blocks of omissions have size $2(f+g+h)-1$.

We leave the further details (with proper induction and finding
the exact period) to the reader.
\end{proof}
\begin{example}\rm
As a more or less random example, take the initial values $[3, 7, 13]$, so 
$f=3$, $g=7$ and $d=13-7=6$.
We find for the difference sequence 
$$D_{3,7,13}=(4, 6, 1^9, 42, \overline{1^9, 44}).$$
From this an initial part of $S_{3,7,13}$ is easily found: 
$$(3, 7, 13, 14, 15, 16, 17, 18, 19, 20, 21, 22, 64, 65, 66, 67, 68, 69, 70, 71,\ldots).$$ 
%72, 73, 117, \ldots).$$
At first it may seem that the repetition starts from the block
$13, 14, \ldots,$ $22$, but it is followed by a block $23, 24, \ldots, 63$
of $41$ omissions, while after the next block of 10 entries $64, 65, \ldots, 73$ there is a block of $43$ omissions. The latter pattern repeats.
\end{example}
\begin{example}\label{ex:124}\rm
The case $S_{1,2,4}$ explicitly excluded from Theorem \ref{thm:gend>1} has
a larger than usual preperiod: $k=13$, while the period length is just 2.
The sequence starts as follows, as an easy computation will verify:
$$S_{1,2,4}=(1, 2, 4, 5, 6, 14, 16, 18, 30, 42, 43, 44, 56, 69, 70, 82, 83, \ldots).$$
It is not difficult to extend the computation to the initial 95 entries: one
finds precisely the pairs $95, 96$, and $108, 109$ and 
$121, 124$, up to $589, 590$ as next elements. To prove that this
holds, one first observes that the holes between two such pairs
consist of 3-sums; take for example the excluded integers
$97, \ldots 107$ between 96 and 108. They can be written
as all sums
$$\{82, 83\}+\{14, 16, 18\}+\{1\}\quad\textrm{and}\quad
\{83\}+\{16, 18\}+\{4,5,6\}$$
of previous entries (with, in fact, one duplication, as
$83+16+6=83+18+4$). Replacing $82$ and $83$ by the next pair
of entries $95, 96$ one finds the omissions between $109$ and $121$,
etcetera. To complete the proof, one needs to verify that all the
pairs listed can \emph{not} be written as 3-sums from previous entries.
We leave this to the readers (or their computers).

Once this tedious task is completed, we may invoke Corollary \ref{corollary}
with $M=65$ to finish the proof of the statements on periodicity of 
this example.
\end{example}
\begin{example}\rm
To show how different sequences can behave for small $g$ relative to $d$,
we mention the case $[1,2,14]$, for which Theorem \ref{thm:gend>1} does
not apply: it turns out that
$S_{1, 2, 14}$ has a preperiod of length $2432$ and a period of length $1144$,
and it becomes periodic modulo 11099.
Compare this to the first `regular' case with $f=2$ and $d=12$: the sequence
$S_{2, 11, 23}$ has preperiod length 16, period 13 and becomes
periodic modulo 82.
\end{example}
The case $d=1$ is slightly different again.
\begin{theorem}\label{thm:d=1}
For every $f\geq 1$ and every $g\geq f+1$ the greedy 3-sumfree sequence $S_{f,g,g+1}$ is ultimately periodic modulo $10g+6f-3$, with preperiod length $g+f+3$ and period length $2g+2f-1$. More explicitly, $D_{f,g,g+1}$ is ultimately periodic in these cases, with preperiod
$$(g-f, 1^{f+g}, 4g+2f-2, 1)$$
and period
$$(1^{f+g-2}, 4g+2f, 1^{f+g-1}, 4g+2f).$$
\end{theorem}
\begin{proof}
We leave the (again rather tedious) proof, that may follow the same structure as that of Theorem \ref{thm:gend>1}, to the reader.
\end{proof}
\begin{example}\rm
As an example, take the initial values $[4,7,8]$, so 
$f=4$, $g=7$ and $d=1$.
Then the difference sequence 
$D_{4,7,8}$ equals
$$(3, 1^{11}, 34, 1, \overline{1^9, 36, 1^{10}, 36}).$$
Hence $S_{4,7,8}$ starts as
$$(4, 7\ldots18, 52\ldots62,98\ldots108,144\ldots153,189\ldots199,\ldots),$$
and the first period (of length 21) properly starts at 53 (not 52!)
and ends at 108.
\end{example}
Taken together, Theorems \ref{thm:d=1} and \ref{thm:gend>1} 
(with Example \ref{ex:124}) furnish a proof for Theorem 
\ref{thm:k=3d>1short} in the Introduction.

\section{Computational results}\label{sec:small}
As is rather conspicuous from Theorems \ref{thm:k=1d=1}, \ref{thm:f=1d>1}, \ref{thm:gend>1} and \ref{thm:d=1}, the general theorems usually exclude finitely
many smaller cases for the parameters.
In this section we describe our efforts to systematically explore such
cases computationally, exploiting Lemma \ref{lemma}. The aim was to
deal with cases $d\leq 25$ left open by Theorem \ref{thm:gend>1}, explicitly
with all pairs $f, g$ such that $1\leq f<g<25$.

Our computations, results of which can be found in 
Table \ref{short} and Table \ref{dat},
consisted of two main parts: first to compute enough terms of $S_{f,g,g+d}$
to get a conjectured value for the pre-period length $k$ and the period
length $p$
(both taken as small as possible) and next to use this in an application
of Lemma \ref{lemma} to {\it prove} that these values indeed give
the correct (pre)periods; compare Section \ref{subsec:periods}. 
All our results can be found in Table \ref{tab:except}.

The tables also list the value $m$, which is the length of the period of
$C_{f,g,g+d}$; as pointed out in the introduction, the
sequences $S_{f,g,g+d}$ will be ultimately periodic modulo $m$.

In Table \ref{dat} we summarize exceptionally large values for
period and preperiod length of $S_{f,g,h}$
that occurred for the 2024 triples $(f, g, h)$ with $1\leq f<g<h$ for which
both $f\leq 25$ and $d=h-g\leq 25$; these include all cases
for which a period or
pre-period length larger than 5000 was found, and also the 12 cases
for which ultimate periodicity could not (yet) be established.
The latter cases are those with values $k=p=-1$ and $m=0$ in the table.
For these 12 cases we have calculated at least 500000 entries of $S_{f,g,h}$.

The complete collection of data establishes a proof for 
Theorem \ref{thm:smallshort}.

\begin{remark}\label{rmk:hiro}\rm
After finishing (a preprint version of) this paper, we were informed
that Hirotaka Shimizu has extended our calculations on the twelve exceptionally
large cases mentioned above enormously: by pushing the sequences to
around $35\cdot 10^6$ entries if necessary, he was able to find periodicity for
7 of the 12 cases; the remaining cases are the sequences with starting
triples $(4,6,17)$, $(4,5,19)$, $(7,9,26)$, $(5,7,26)$ and $(6,7,27)$.
We are grateful to Hirotaka to allow us to include his findings here
\cite{hiro}.
\end{remark}

{\tiny
\begin{table}
\parbox{.48\linewidth}{
        \centering
\begin{tabular}{|c|c|c|c|c|c|}
\hline
$d$ & $(f, g, h)$ & $k$ & $p$ &$m$\\
\hline
4 & ( 1, 2, 6 ) & 6 & 6 & 28\\
\hline
5 & ( 1, 2, 7 ) & 8 & 8 & 44\\
5 & ( 1, 3, 8 ) & 8 & 6 & 38\\
5 & ( 2, 3, 8 ) & 208 & 237 & 2215\\
\hline
6 & ( 1, 2, 8 ) & 8 & 8 & 41\\
6 & ( 1, 3, 9 ) & 9 & 7 & 41\\
6 & ( 1, 4, 10 ) & 9 & 6 & 48\\
6 & ( 2, 3, 9 ) & 11 & 7 & 44\\
6 & ( 2, 4, 10 ) & 9 & 6 & 43\\
6 & ( 3, 4, 10 ) & 9 & 7 & 38\\
\hline
7 & ( 1, 2, 9 ) & 6 & 6 & 37\\
7 & ( 1, 3, 10 ) & 23 & 21 & 193\\
7 & ( 1, 4, 11 ) & 10 & 7 & 51\\
7 & ( 1, 5, 12 ) & 26 & 32 & 321\\
7 & ( 2, 3, 10 ) & 8 & 9 & 57\\
7 & ( 2, 4, 11 ) & 12 & 103 & 928\\
7 & ( 2, 5, 12 ) & 10 & 7 & 51\\
7 & ( 3, 4, 11 ) & 10 & 7 & 49\\
7 & ( 3, 5, 12 ) & 10 & 8 & 45\\
7 & ( 4, 5, 12 ) & 11 & 9 & 48\\
\hline
8 & ( 1, 2, 10 ) & 11 & 11 & 96\\
8 & ( 1, 3, 11 ) & 21 & 11 & 101\\
8 & ( 1, 4, 12 ) & 43 & 60 & 518\\
8 & ( 1, 5, 13 ) & 32 & 10 & 117\\
8 & ( 1, 6, 14 ) & 29 & 51 & 502\\
8 & ( 2, 3, 11 ) & 40 & 9 & 61\\
8 & ( 2, 4, 12 ) & 11 & 7 & 57\\
8 & ( 2, 5, 13 ) & 38 & 38 & 362\\
8 & ( 2, 6, 14 ) & 11 & 8 & 59\\
8 & ( 3, 4, 12 ) & 13 & 110 & 891\\
8 & ( 3, 5, 13 ) & 11 & 8 & 57\\
8 & ( 3, 6, 14 ) & 11 & 9 & 52\\
8 & ( 4, 5, 13 ) & 11 & 9 & 50\\
8 & ( 4, 6, 14 ) & 12 & 10 & 55\\
8 & ( 5, 6, 14 ) & 13 & 11 & 58\\
\hline
9 & ( 1, 2, 11 ) & 51 & 12 & 114\\
9 & ( 1, 3, 12 ) & 23 & 45 & 423\\
9 & ( 1, 4, 13 ) & 24 & 27 & 249\\
9 & ( 1, 5, 14 ) & 61 & 83 & 729\\
9 & ( 1, 6, 15 ) & 37 & 12 & 135\\
9 & ( 1, 7, 16 ) & 32 & 74 & 719\\
9 & ( 2, 3, 12 ) & 14 & 10 & 76\\
\hline
\end{tabular}
}
\hfill
\parbox{.48\linewidth}{
        \centering
\begin{tabular}{|c|c|c|c|c|c|}
\hline
$d$ & $(f, g, h)$ & $k$ & $p$ &$m$\\
\hline
9 & ( 2, 4, 13 ) & 49 & 72 & 585\\
9 & ( 2, 5, 14 ) & 41 & 125 & 1237\\
9 & ( 2, 6, 15 ) & 145 & 250 & 2433\\
9 & ( 2, 7, 16 ) & 12 & 9 & 67\\
9 & ( 3, 4, 13 ) & 12 & 42 & 330\\
9 & ( 3, 5, 14 ) & 14 & 969 & 8453\\
9 & ( 3, 6, 15 ) & 12 & 9 & 65\\
9 & ( 3, 7, 16 ) & 12 & 10 & 59\\
9 & ( 4, 5, 14 ) & 12 & 9 & 63\\
9 & ( 4, 6, 15 ) & 12 & 10 & 57\\
9 & ( 4, 7, 16 ) & 13 & 11 & 62\\
9 & ( 5, 6, 15 ) & 13 & 11 & 60\\
9 & ( 5, 7, 16 ) & 14 & 12 & 65\\
9 & ( 6, 7, 16 ) & 15 & 13 & 68\\
\hline
10 & ( 1, 2, 12 ) & 11 & 10 & 64\\
10 & ( 1, 3, 13 ) & 42 & 46 & 442\\
10 & ( 1, 4, 14 ) & 24 & 14 & 129\\
10 & ( 1, 5, 15 ) & 31 & 31 & 287\\
10 & ( 1, 6, 16 ) & 224 & 377 & 3298\\
10 & ( 1, 7, 17 ) & 42 & 14 & 153\\
10 & ( 1, 8, 18 ) & 35 & 101 & 972\\
10 & ( 2, 3, 13 ) & 20 & 19 & 148\\
10 & ( 2, 4, 14 ) & 28 & 31 & 279\\
10 & ( 2, 5, 15 ) & 124 & 26 & 220\\
10 & ( 2, 6, 16 ) & 40 & 14 & 149\\
10 & ( 2, 7, 17 ) & 187 & 86 & 786\\
10 & ( 2, 8, 18 ) & 13 & 10 & 75\\
10 & ( 3, 4, 14 ) & 14 & 10 & 81\\
10 & ( 3, 5, 15 ) & 13 & 9 & 73\\
10 & ( 3, 6, 16 ) & 55 & 69 & 612\\
10 & ( 3, 7, 17 ) & 13 & 10 & 73\\
10 & ( 3, 8, 18 ) & 13 & 11 & 66\\
10 & ( 4, 5, 15 ) & 15 & 178 & 1433\\
10 & ( 4, 6, 16 ) & 13 & 10 & 71\\
10 & ( 4, 7, 17 ) & 13 & 11 & 64\\
10 & ( 4, 8, 18 ) & 14 & 12 & 69\\
10 & ( 5, 6, 16 ) & 13 & 11 & 62\\
10 & ( 5, 7, 17 ) & 14 & 12 & 67\\
10 & ( 5, 8, 18 ) & 15 & 13 & 72\\
10 & ( 6, 7, 17 ) & 15 & 13 & 70\\
10 & ( 6, 8, 18 ) & 16 & 14 & 75\\
10 & ( 7, 8, 18 ) & 17 & 15 & 78\\
\hline
\end{tabular}
}
\bigskip
\caption{Small cases: all cases with $d\leq 10$ for which neither Theorem \ref{thm:k=3d>1short} nor Conjecture \ref{conj:except} does apply. See Table \ref{dat} for further explanation of the data.\label{short}}
\end{table}

}

%{\small
{\tiny
\begin{table}[ht]\label{tab:small}
\begin{tabular}{|c|c|c|c|c|c|c|c|c|}
\hline
$(f, g, h)$ & $k$ & $p$ &$m$ && $(f, g, h)$ & $k$ & $p$ &$m$\\
\hline
 ( 4, 6, 17 ) &  $-1$ &  $-1$ &  0 &&
 ( 9, 11, 32 ) &  27 &  6642 &  54721 \\
 ( 4, 8, 21 ) &  14620 &  6651 &  60839 &&
 ( 1, 14, 36 ) &  11711 &  16276 &  150981\\
 ( 1, 2, 16 ) &  5986 &  11817 &  122842 &&
 ( 2, 17, 39 ) &  5489 &  4907 &  42958\\
 ( 4, 5, 19 ) &  $-1$ &  $-1$ &  0 &&
 ( 3, 15, 37 ) &  28900 &  16134 &  139289\\
 ( 1, 9, 24 ) &  $-1$ &  $-1$ &  0 &&
 ( 4, 5, 27 ) &  4779 &  11302 &  105810\\
 ( 4, 10, 25 ) &  11756 &  1056 &  10044 &&
 ( 6, 8, 30 ) &  8874 &  22170 &  185344\\
 ( 3, 4, 20 ) &  6556 &  30223 &  269136 &&
 ( 8, 13, 35 ) &  8324 &  6024 &  53145\\
\hline
\end{tabular}
\bigskip
\caption{Large cases: all $(f,g,g+d)$ with $d,f\leq 25$ and $f+g\leq d$ for which neither Theorem \ref{thm:k=3d>1short} nor Conjecture \ref{conj:except} does apply and for which either the length $k$ of the preperiod or $p$ of the period of $S_{f,g,g+d}$ exceeds 5000. An entry with $k=p=-1$ signifies that no periodicity was found at 500000 terms (but see Remark \ref{rmk:hiro}). The final column lists the modulus of periodicity.\label{dat}}
\end{table}

}

{\small
{\tiny
\begin{table}
\parbox{.45\linewidth}{
        \centering
\begin{tabular}{|c|c|c|c|c|c|}
\hline
$d$ & $(f, g, h)$ & $k$ & $p$ &$m$\\
\hline
4 & ( 1, 2, 6 ) & 6 & 6 & 28 \\
\hline
5 & ( 1, 2, 7 ) & 8 & 8 & 44 \\
5 & ( 1, 3, 8 ) & 8 & 6 & 38 \\
5 & ( 2, 3, 8 ) & 208 & 237 & 2215 \\
\hline
6 & ( 1, 2, 8 ) & 8 & 8 & 41 \\
6 & ( 1, 3, 9 ) & 9 & 7 & 41 \\
6 & ( 1, 4, 10 ) & 9 & 6 & 48 \\
6 & ( 2, 3, 9 ) & 11 & 7 & 44 \\
6 & ( 2, 4, 10 ) & 9 & 6 & 43 \\
6 & ( 3, 4, 10 ) & 9 & 7 & 38 \\
\hline
7 & ( 1, 2, 9 ) & 6 & 6 & 37 \\
7 & ( 1, 3, 10 ) & 23 & 21 & 193 \\
7 & ( 1, 4, 11 ) & 10 & 7 & 51 \\
7 & ( 1, 5, 12 ) & 26 & 32 & 321 \\
7 & ( 2, 3, 10 ) & 8 & 9 & 57 \\
7 & ( 2, 4, 11 ) & 12 & 103 & 928 \\
7 & ( 2, 5, 12 ) & 10 & 7 & 51 \\
7 & ( 3, 4, 11 ) & 10 & 7 & 49 \\
7 & ( 3, 5, 12 ) & 10 & 8 & 45 \\
7 & ( 4, 5, 12 ) & 11 & 9 & 48 \\
\hline
8 & ( 1, 2, 10 ) & 11 & 11 & 96 \\
8 & ( 1, 3, 11 ) & 21 & 11 & 101 \\
8 & ( 1, 4, 12 ) & 43 & 60 & 518 \\
8 & ( 1, 5, 13 ) & 32 & 10 & 117 \\
8 & ( 1, 6, 14 ) & 29 & 51 & 502 \\
8 & ( 2, 3, 11 ) & 40 & 9 & 61 \\
8 & ( 2, 4, 12 ) & 11 & 7 & 57 \\
8 & ( 2, 5, 13 ) & 38 & 38 & 362 \\
8 & ( 2, 6, 14 ) & 11 & 8 & 59 \\
8 & ( 3, 4, 12 ) & 13 & 110 & 891 \\
8 & ( 3, 5, 13 ) & 11 & 8 & 57 \\
8 & ( 3, 6, 14 ) & 11 & 9 & 52 \\
8 & ( 4, 5, 13 ) & 11 & 9 & 50 \\
8 & ( 4, 6, 14 ) & 12 & 10 & 55 \\
8 & ( 5, 6, 14 ) & 13 & 11 & 58 \\
\hline
9 & ( 1, 2, 11 ) & 51 & 12 & 114 \\
9 & ( 1, 3, 12 ) & 23 & 45 & 423 \\
9 & ( 1, 4, 13 ) & 24 & 27 & 249 \\
9 & ( 1, 5, 14 ) & 61 & 83 & 729 \\
9 & ( 1, 6, 15 ) & 37 & 12 & 135 \\
9 & ( 1, 7, 16 ) & 32 & 74 & 719 \\
9 & ( 2, 3, 12 ) & 14 & 10 & 76 \\
\hline
\end{tabular}
%\caption{dit\label{labdit}}
}
\hfill
\parbox{.45\linewidth}{
        \centering
\begin{tabular}{|c|c|c|c|c|c|}
\hline
$d$ & $(f, g, h)$ & $k$ & $p$ &$m$\\
\hline
9 & ( 2, 4, 13 ) & 49 & 72 & 585 \\
9 & ( 2, 5, 14 ) & 41 & 125 & 1237 \\
9 & ( 2, 6, 15 ) & 145 & 250 & 2433 \\
9 & ( 2, 7, 16 ) & 12 & 9 & 67 \\
9 & ( 3, 4, 13 ) & 12 & 42 & 330 \\
9 & ( 3, 5, 14 ) & 14 & 969 & 8453 \\
9 & ( 3, 6, 15 ) & 12 & 9 & 65 \\
9 & ( 3, 7, 16 ) & 12 & 10 & 59 \\
9 & ( 4, 5, 14 ) & 12 & 9 & 63 \\
9 & ( 4, 6, 15 ) & 12 & 10 & 57 \\
9 & ( 4, 7, 16 ) & 13 & 11 & 62 \\
9 & ( 5, 6, 15 ) & 13 & 11 & 60 \\
9 & ( 5, 7, 16 ) & 14 & 12 & 65 \\
9 & ( 6, 7, 16 ) & 15 & 13 & 68 \\
\hline
10 & ( 1, 2, 12 ) & 11 & 10 & 64 \\
10 & ( 1, 3, 13 ) & 42 & 46 & 442 \\
10 & ( 1, 4, 14 ) & 24 & 14 & 129 \\
10 & ( 1, 5, 15 ) & 31 & 31 & 287 \\
10 & ( 1, 6, 16 ) & 224 & 377 & 3298 \\
10 & ( 1, 7, 17 ) & 42 & 14 & 153 \\
10 & ( 1, 8, 18 ) & 35 & 101 & 972 \\
10 & ( 2, 3, 13 ) & 20 & 19 & 148 \\
10 & ( 2, 4, 14 ) & 28 & 31 & 279 \\
10 & ( 2, 5, 15 ) & 124 & 26 & 220 \\
10 & ( 2, 6, 16 ) & 40 & 14 & 149 \\
10 & ( 2, 7, 17 ) & 187 & 86 & 786 \\
10 & ( 2, 8, 18 ) & 13 & 10 & 75 \\
10 & ( 3, 4, 14 ) & 14 & 10 & 81 \\
10 & ( 3, 5, 15 ) & 13 & 9 & 73 \\
10 & ( 3, 6, 16 ) & 55 & 69 & 612 \\
10 & ( 3, 7, 17 ) & 13 & 10 & 73 \\
10 & ( 3, 8, 18 ) & 13 & 11 & 66 \\
10 & ( 4, 5, 15 ) & 15 & 178 & 1433 \\
10 & ( 4, 6, 16 ) & 13 & 10 & 71 \\
10 & ( 4, 7, 17 ) & 13 & 11 & 64 \\
10 & ( 4, 8, 18 ) & 14 & 12 & 69 \\
10 & ( 5, 6, 16 ) & 13 & 11 & 62 \\
10 & ( 5, 7, 17 ) & 14 & 12 & 67 \\
10 & ( 5, 8, 18 ) & 15 & 13 & 72 \\
10 & ( 6, 7, 17 ) & 15 & 13 & 70 \\
10 & ( 6, 8, 18 ) & 16 & 14 & 75 \\
10 & ( 7, 8, 18 ) & 17 & 15 & 78 \\
\hline
\end{tabular}
%\caption{dat\label{labdat}}
} 
\end{table}

\begin{table}
\parbox{.45\linewidth}{
        \centering
\begin{tabular}{|c|c|c|c|c|c|}
\hline
$d$ & $(f, g, h)$ & $k$ & $p$ &$m$\\
\hline
11 & ( 1, 2, 13 ) & 13 & 12 & 80 \\
11 & ( 1, 3, 14 ) & 12 & 11 & 117 \\
11 & ( 1, 4, 15 ) & 24 & 29 & 267 \\
11 & ( 1, 5, 16 ) & 21 & 21 & 163 \\
11 & ( 1, 6, 17 ) & 57 & 80 & 726 \\
11 & ( 1, 7, 18 ) & 65 & 144 & 1259 \\
11 & ( 1, 8, 19 ) & 47 & 16 & 171 \\
11 & ( 1, 9, 20 ) & 38 & 132 & 1261 \\
11 & ( 2, 3, 14 ) & 16 & 17 & 154 \\
11 & ( 2, 4, 15 ) & 21 & 21 & 159 \\
11 & ( 2, 5, 16 ) & 73 & 35 & 317 \\
11 & ( 2, 6, 17 ) & 261 & 430 & 3619 \\
11 & ( 2, 7, 18 ) & 45 & 16 & 167 \\
11 & ( 2, 8, 19 ) & 249 & 115 & 1053 \\
11 & ( 2, 9, 20 ) & 14 & 11 & 83 \\
11 & ( 3, 4, 15 ) & 38 & 55 & 438 \\
11 & ( 3, 5, 16 ) & 140 & 30 & 242 \\
11 & ( 3, 6, 17 ) & 51 & 206 & 1947 \\
11 & ( 3, 7, 18 ) & 194 & 74 & 683 \\
11 & ( 3, 8, 19 ) & 14 & 11 & 81 \\
11 & ( 3, 9, 20 ) & 14 & 12 & 73 \\
11 & ( 4, 5, 16 ) & 14 & 72 & 572 \\
11 & ( 4, 6, 17 ) & -1 & -1 & -1 \\
11 & ( 4, 7, 18 ) & 14 & 11 & 79 \\
11 & ( 4, 8, 19 ) & 14 & 12 & 71 \\
11 & ( 4, 9, 20 ) & 15 & 13 & 76 \\
11 & ( 5, 6, 17 ) & 14 & 11 & 77 \\
11 & ( 5, 7, 18 ) & 14 & 12 & 69 \\
11 & ( 5, 8, 19 ) & 15 & 13 & 74 \\
11 & ( 5, 9, 20 ) & 16 & 14 & 79 \\
11 & ( 6, 7, 18 ) & 15 & 13 & 72 \\
11 & ( 6, 8, 19 ) & 16 & 14 & 77 \\
11 & ( 6, 9, 20 ) & 17 & 15 & 82 \\
11 & ( 7, 8, 19 ) & 17 & 15 & 80 \\
11 & ( 7, 9, 20 ) & 18 & 16 & 85 \\
11 & ( 8, 9, 20 ) & 19 & 17 & 88 \\
\hline
12 & ( 1, 2, 14 ) & 2432 & 1144 & 11099 \\
12 & ( 1, 3, 15 ) & 51 & 13 & 134 \\
12 & ( 1, 4, 16 ) & 13 & 15 & 138 \\
12 & ( 1, 5, 17 ) & 21 & 11 & 84 \\
12 & ( 1, 6, 18 ) & 31 & 37 & 343 \\
12 & ( 1, 7, 19 ) & 124 & 132 & 1165 \\
12 & ( 1, 8, 20 ) & 109 & 180 & 1578 \\
12 & ( 1, 9, 21 ) & 52 & 18 & 189 \\
12 & ( 1, 10, 22 ) & 41 & 167 & 1586 \\
12 & ( 2, 3, 15 ) & 16 & 17 & 161 \\
12 & ( 2, 4, 16 ) & 21 & 11 & 82 \\
12 & ( 2, 5, 17 ) & 474 & 477 & 4313 \\
12 & ( 2, 6, 18 ) & 106 & 136 & 1151 \\
12 & ( 2, 7, 19 ) & 115 & 162 & 1368 \\
\hline
\end{tabular}
%\caption{dit\label{labdit}}
}
\hfill
\parbox{.45\linewidth}{
        \centering
\begin{tabular}{|c|c|c|c|c|c|}
\hline
$d$ & $(f, g, h)$ & $k$ & $p$ &$m$\\
\hline
12 & ( 2, 8, 20 ) & 50 & 18 & 185 \\
12 & ( 2, 9, 21 ) & 251 & 322 & 3006 \\
12 & ( 2, 10, 22 ) & 15 & 12 & 91 \\
12 & ( 3, 4, 16 ) & 242 & 1971 & 16749 \\
12 & ( 3, 5, 17 ) & 63 & 21 & 174 \\
12 & ( 3, 6, 18 ) & 30 & 77 & 626 \\
12 & ( 3, 7, 19 ) & 48 & 18 & 181 \\
12 & ( 3, 8, 20 ) & 1021 & 276 & 2653 \\
12 & ( 3, 9, 21 ) & 15 & 12 & 89 \\
12 & ( 3, 10, 22 ) & 15 & 13 & 80 \\
12 & ( 4, 5, 17 ) & 158 & 33 & 264 \\
12 & ( 4, 6, 18 ) & 15 & 11 & 89 \\
12 & ( 4, 7, 19 ) & 78 & 94 & 832 \\
12 & ( 4, 8, 20 ) & 15 & 12 & 87 \\
12 & ( 4, 9, 21 ) & 15 & 13 & 78 \\
12 & ( 4, 10, 22 ) & 16 & 14 & 83 \\
12 & ( 5, 6, 18 ) & 17 & 262 & 2103 \\
12 & ( 5, 7, 19 ) & 15 & 12 & 85 \\
12 & ( 5, 8, 20 ) & 15 & 13 & 76 \\
12 & ( 5, 9, 21 ) & 16 & 14 & 81 \\
12 & ( 5, 10, 22 ) & 17 & 15 & 86 \\
12 & ( 6, 7, 19 ) & 15 & 13 & 74 \\
12 & ( 6, 8, 20 ) & 16 & 14 & 79 \\
12 & ( 6, 9, 21 ) & 17 & 15 & 84 \\
12 & ( 6, 10, 22 ) & 18 & 16 & 89 \\
12 & ( 7, 8, 20 ) & 17 & 15 & 82 \\
12 & ( 7, 9, 21 ) & 18 & 16 & 87 \\
12 & ( 7, 10, 22 ) & 19 & 17 & 92 \\
12 & ( 8, 9, 21 ) & 19 & 17 & 90 \\
12 & ( 8, 10, 22 ) & 20 & 18 & 95 \\
12 & ( 9, 10, 22 ) & 21 & 19 & 98 \\
\hline
13 & ( 1, 2, 15 ) & 41 & 153 & 1668 \\
13 & ( 1, 3, 16 ) & 13 & 28 & 295 \\
13 & ( 1, 4, 17 ) & 13 & 15 & 142 \\
13 & ( 1, 5, 18 ) & 21 & 23 & 173 \\
13 & ( 1, 6, 19 ) & 24 & 25 & 193 \\
13 & ( 1, 7, 20 ) & 39 & 41 & 381 \\
13 & ( 1, 8, 21 ) & 88 & 33 & 295 \\
13 & ( 1, 9, 22 ) & 139 & 162 & 1425 \\
13 & ( 1, 10, 23 ) & 57 & 20 & 207 \\
13 & ( 1, 11, 24 ) & 44 & 206 & 1947 \\
13 & ( 2, 3, 16 ) & 16 & 17 & 168 \\
13 & ( 2, 4, 17 ) & 21 & 23 & 169 \\
13 & ( 2, 5, 18 ) & 24 & 25 & 189 \\
13 & ( 2, 6, 19 ) & 35 & 41 & 373 \\
13 & ( 2, 7, 20 ) & 138 & 148 & 1261 \\
13 & ( 2, 8, 21 ) & 119 & 200 & 1701 \\
13 & ( 2, 9, 22 ) & 55 & 20 & 203 \\
13 & ( 2, 10, 23 ) & 344 & 443 & 4139 \\
13 & ( 2, 11, 24 ) & 16 & 13 & 99 \\
\hline
\end{tabular}
%\caption{dat\label{labdat}}
} 
\end{table}

\begin{table}
\parbox{.45\linewidth}{
        \centering
\begin{tabular}{|c|c|c|c|c|c|}
\hline
$d$ & $(f, g, h)$ & $k$ & $p$ &$m$\\
\hline
13 & ( 3, 4, 17 ) & 16 & 117 & 936 \\
13 & ( 3, 5, 18 ) & 539 & 543 & 4699 \\
13 & ( 3, 6, 19 ) & 117 & 153 & 1248 \\
13 & ( 3, 7, 20 ) & 127 & 180 & 1477 \\
13 & ( 3, 8, 21 ) & 53 & 20 & 199 \\
13 & ( 3, 9, 22 ) & 778 & 379 & 3513 \\
13 & ( 3, 10, 23 ) & 16 & 13 & 97 \\
\hline
\end{tabular}
%\caption{dit\label{labdit}}
}
\hfill
\parbox{.45\linewidth}{
        \centering
\begin{tabular}{|c|c|c|c|c|c|}
\hline
$d$ & $(f, g, h)$ & $k$ & $p$ &$m$\\
\hline
14 & ( 3, 4, 18 ) & 20 & 130 & 1063 \\
14 & ( 3, 5, 19 ) & 26 & 27 & 205 \\
14 & ( 3, 6, 20 ) & 44 & 45 & 403 \\
14 & ( 3, 7, 21 ) & 1406 & 1017 & 8453 \\
14 & ( 3, 8, 22 ) & 46 & 116 & 964 \\
14 & ( 3, 9, 23 ) & 58 & 22 & 217 \\
14 & ( 3, 10, 24 ) & 371 & 176 & 1576 \\
\hline
\end{tabular}
%\caption{dat\label{labdat}}
} 
\end{table}

\begin{table}
\parbox{.45\linewidth}{
        \centering
\begin{tabular}{|c|c|c|c|c|c|}
\hline
$d$ & $(f, g, h)$ & $k$ & $p$ &$m$\\
\hline
15 & ( 2, 3, 18 ) & 14 & 40 & 380 \\
15 & ( 2, 4, 19 ) & 14 & 12 & 89 \\
15 & ( 2, 5, 20 ) & 24 & 27 & 199 \\
15 & ( 2, 6, 21 ) & 27 & 29 & 219 \\
15 & ( 2, 7, 22 ) & 38 & 47 & 429 \\
15 & ( 2, 8, 23 ) & 2312 & 2622 & 23320 \\
15 & ( 2, 9, 24 ) & 230 & 556 & 4815 \\
\hline
\end{tabular}
%\caption{dit\label{labdit}}
}
\hfill
\parbox{.45\linewidth}{
        \centering
\begin{tabular}{|c|c|c|c|c|c|}
\hline
$d$ & $(f, g, h)$ & $k$ & $p$ &$m$\\
\hline
15 & ( 9, 11, 26 ) & 22 & 20 & 109 \\
15 & ( 9, 12, 27 ) & 23 & 21 & 114 \\
15 & ( 9, 13, 28 ) & 24 & 22 & 119 \\
15 & ( 10, 11, 26 ) & 23 & 21 & 112 \\
15 & ( 10, 12, 27 ) & 24 & 22 & 117 \\
15 & ( 10, 13, 28 ) & 25 & 23 & 122 \\
15 & ( 11, 12, 27 ) & 25 & 23 & 120 \\
\hline
\end{tabular}
%\caption{dat\label{labdat}}
} 
\end{table}

\begin{table}
\parbox{.45\linewidth}{
        \centering
\begin{tabular}{|c|c|c|c|c|c|}
\hline
$d$ & $(f, g, h)$ & $k$ & $p$ &$m$\\
\hline
16 & ( 5, 8, 24 ) & 54 & 137 & 1094 \\
16 & ( 5, 9, 25 ) & 64 & 26 & 245 \\
16 & ( 5, 10, 26 ) & 2326 & 585 & 5444 \\
16 & ( 5, 11, 27 ) & 19 & 16 & 117 \\
16 & ( 5, 12, 28 ) & 19 & 17 & 104 \\
16 & ( 5, 13, 29 ) & 20 & 18 & 109 \\
16 & ( 5, 14, 30 ) & 21 & 19 & 114 \\
\hline
\end{tabular}
%\caption{dit\label{labdit}}
}
\hfill
\parbox{.45\linewidth}{
        \centering
\begin{tabular}{|c|c|c|c|c|c|}
\hline
$d$ & $(f, g, h)$ & $k$ & $p$ &$m$\\
\hline
17 & ( 1, 12, 29 ) & 274 & 377 & 3352 \\
17 & ( 1, 13, 30 ) & 249 & 311 & 2751 \\
17 & ( 1, 14, 31 ) & 77 & 28 & 279 \\
17 & ( 1, 15, 32 ) & 56 & 402 & 3751 \\
17 & ( 2, 3, 20 ) & 67 & 140 & 1202 \\
17 & ( 2, 4, 21 ) & 14 & 12 & 93 \\
17 & ( 2, 5, 22 ) & 16 & 14 & 104 \\
\hline
\end{tabular}
%\caption{dat\label{labdat}}
} 
\end{table}

\begin{table}
\parbox{.45\linewidth}{
        \centering
\begin{tabular}{|c|c|c|c|c|c|}
\hline
$d$ & $(f, g, h)$ & $k$ & $p$ &$m$\\
\hline
17 & ( 6, 10, 27 ) & 487 & 874 & 7734 \\
17 & ( 6, 11, 28 ) & 20 & 17 & 123 \\
17 & ( 6, 12, 29 ) & 20 & 18 & 109 \\
17 & ( 6, 13, 30 ) & 21 & 19 & 114 \\
17 & ( 6, 14, 31 ) & 22 & 20 & 119 \\
17 & ( 6, 15, 32 ) & 23 & 21 & 124 \\
17 & ( 7, 8, 25 ) & 20 & 210 & 1682 \\
\hline
\end{tabular}
%\caption{dit\label{labdit}}
}
\hfill
\parbox{.45\linewidth}{
        \centering
\begin{tabular}{|c|c|c|c|c|c|}
\hline
$d$ & $(f, g, h)$ & $k$ & $p$ &$m$\\
\hline
18 & ( 2, 3, 21 ) & 69 & 65 & 606 \\
18 & ( 2, 4, 22 ) & 15 & 12 & 103 \\
18 & ( 2, 5, 23 ) & 16 & 14 & 106 \\
18 & ( 2, 6, 24 ) & 18 & 16 & 117 \\
18 & ( 2, 7, 25 ) & 30 & 17 & 127 \\
18 & ( 2, 8, 26 ) & 32 & 35 & 269 \\
18 & ( 2, 9, 27 ) & 45 & 57 & 523 \\
\hline
\end{tabular}
%\caption{dat\label{labdat}}
} 
\end{table}

\begin{table}
\parbox{.45\linewidth}{
        \centering
\begin{tabular}{|c|c|c|c|c|c|}
\hline
$d$ & $(f, g, h)$ & $k$ & $p$ &$m$\\
\hline
18 & ( 6, 14, 32 ) & 22 & 20 & 121 \\
18 & ( 6, 15, 33 ) & 23 & 21 & 126 \\
18 & ( 6, 16, 34 ) & 24 & 22 & 131 \\
18 & ( 7, 8, 26 ) & 76 & 158 & 1224 \\
18 & ( 7, 9, 27 ) & 21 & 17 & 137 \\
18 & ( 7, 10, 28 ) & 186 & 214 & 1826 \\
18 & ( 7, 11, 29 ) & 21 & 18 & 129 \\
\hline
\end{tabular}
%\caption{dit\label{labdit}}
}
\hfill
\parbox{.45\linewidth}{
        \centering
\begin{tabular}{|c|c|c|c|c|c|}
\hline
$d$ & $(f, g, h)$ & $k$ & $p$ &$m$\\
\hline
19 & ( 1, 11, 30 ) & 55 & 61 & 569 \\
19 & ( 1, 12, 31 ) & -1 & -1 & 0 \\
19 & ( 1, 13, 32 ) & 132 & 52 & 448 \\
19 & ( 1, 14, 33 ) & 855 & 494 & 4406 \\
19 & ( 1, 15, 34 ) & 845 & 544 & 4819 \\
19 & ( 1, 16, 35 ) & 87 & 32 & 315 \\
19 & ( 1, 17, 36 ) & 62 & 524 & 4869 \\
\hline
\end{tabular}
%\caption{dat\label{labdat}}
} 
\end{table}

\begin{table}
\parbox{.45\linewidth}{
        \centering
\begin{tabular}{|c|c|c|c|c|c|}
\hline
$d$ & $(f, g, h)$ & $k$ & $p$ &$m$\\
\hline
19 & ( 5, 14, 33 ) & 22 & 19 & 141 \\
19 & ( 5, 15, 34 ) & 22 & 20 & 125 \\
19 & ( 5, 16, 35 ) & 23 & 21 & 130 \\
19 & ( 5, 17, 36 ) & 24 & 22 & 135 \\
19 & ( 6, 7, 26 ) & 261 & 65 & 549 \\
19 & ( 6, 8, 27 ) & 250 & 308 & 2413 \\
19 & ( 6, 9, 28 ) & 76 & 235 & 1870 \\
\hline
\end{tabular}
%\caption{dit\label{labdit}}
}
\hfill
\parbox{.45\linewidth}{
        \centering
\begin{tabular}{|c|c|c|c|c|c|}
\hline
$d$ & $(f, g, h)$ & $k$ & $p$ &$m$\\
\hline
19 & ( 12, 15, 34 ) & 29 & 27 & 146 \\
19 & ( 12, 16, 35 ) & 30 & 28 & 151 \\
19 & ( 12, 17, 36 ) & 31 & 29 & 156 \\
19 & ( 13, 14, 33 ) & 29 & 27 & 144 \\
19 & ( 13, 15, 34 ) & 30 & 28 & 149 \\
19 & ( 13, 16, 35 ) & 31 & 29 & 154 \\
19 & ( 13, 17, 36 ) & 32 & 30 & 159 \\
\hline
\end{tabular}
%\caption{dat\label{labdat}}
} 
\end{table}

\begin{table}
\parbox{.45\linewidth}{
        \centering
\begin{tabular}{|c|c|c|c|c|c|}
\hline
$d$ & $(f, g, h)$ & $k$ & $p$ &$m$\\
\hline
20 & ( 3, 15, 35 ) & 88 & 34 & 325 \\
20 & ( 3, 16, 36 ) & 582 & 37 & 335 \\
20 & ( 3, 17, 37 ) & 23 & 20 & 153 \\
20 & ( 3, 18, 38 ) & 23 & 21 & 136 \\
20 & ( 4, 5, 25 ) & 444 & 1062 & 9267 \\
20 & ( 4, 6, 26 ) & 33 & 19 & 138 \\
20 & ( 4, 7, 27 ) & 35 & 39 & 291 \\
\hline
\end{tabular}
%\caption{dit\label{labdit}}
}
\hfill
\parbox{.45\linewidth}{
        \centering
\begin{tabular}{|c|c|c|c|c|c|}
\hline
$d$ & $(f, g, h)$ & $k$ & $p$ &$m$\\
\hline
20 & ( 8, 12, 32 ) & 23 & 20 & 143 \\
20 & ( 8, 13, 33 ) & 23 & 21 & 126 \\
20 & ( 8, 14, 34 ) & 24 & 22 & 131 \\
20 & ( 8, 15, 35 ) & 25 & 23 & 136 \\
20 & ( 8, 16, 36 ) & 26 & 24 & 141 \\
20 & ( 8, 17, 37 ) & 27 & 25 & 146 \\
20 & ( 8, 18, 38 ) & 28 & 26 & 151 \\
\hline
\end{tabular}
%\caption{dat\label{labdat}}
} 
\end{table}

\begin{table}
\parbox{.45\linewidth}{
        \centering
\begin{tabular}{|c|c|c|c|c|c|}
\hline
$d$ & $(f, g, h)$ & $k$ & $p$ &$m$\\
\hline
21 & ( 1, 7, 28 ) & 18 & 16 & 125 \\
21 & ( 1, 8, 29 ) & 20 & 18 & 136 \\
21 & ( 1, 9, 30 ) & 33 & 39 & 293 \\
21 & ( 1, 10, 31 ) & 36 & 41 & 313 \\
21 & ( 1, 11, 32 ) & 37 & 41 & 323 \\
21 & ( 1, 12, 33 ) & 52 & 67 & 625 \\
21 & ( 1, 13, 34 ) & 8169 & 1293 & 12001 \\
\hline
\end{tabular}
%\caption{dit\label{labdit}}
}
\hfill
\parbox{.45\linewidth}{
        \centering
\begin{tabular}{|c|c|c|c|c|c|}
\hline
$d$ & $(f, g, h)$ & $k$ & $p$ &$m$\\
\hline
21 & ( 4, 17, 38 ) & 24 & 21 & 159 \\
21 & ( 4, 18, 39 ) & 24 & 22 & 141 \\
21 & ( 4, 19, 40 ) & 25 & 23 & 146 \\
21 & ( 5, 6, 27 ) & 24 & 357 & 2864 \\
21 & ( 5, 7, 28 ) & 37 & 41 & 307 \\
21 & ( 5, 8, 29 ) & 1682 & 1218 & 10435 \\
21 & ( 5, 9, 30 ) & 276 & 69 & 613 \\
\hline
\end{tabular}
%\caption{dat\label{labdat}}
} 
\end{table}

\begin{table}
\parbox{.45\linewidth}{
        \centering
\begin{tabular}{|c|c|c|c|c|c|}
\hline
$d$ & $(f, g, h)$ & $k$ & $p$ &$m$\\
\hline
21 & ( 9, 15, 36 ) & 26 & 24 & 141 \\
21 & ( 9, 16, 37 ) & 27 & 25 & 146 \\
21 & ( 9, 17, 38 ) & 28 & 26 & 151 \\
21 & ( 9, 18, 39 ) & 29 & 27 & 156 \\
21 & ( 9, 19, 40 ) & 30 & 28 & 161 \\
21 & ( 10, 11, 32 ) & 24 & 21 & 147 \\
21 & ( 10, 12, 33 ) & 24 & 22 & 129 \\
\hline
\end{tabular}
%\caption{dit\label{labdit}}
}
\hfill
\parbox{.45\linewidth}{
        \centering
\begin{tabular}{|c|c|c|c|c|c|}
\hline
$d$ & $(f, g, h)$ & $k$ & $p$ &$m$\\
\hline
22 & ( 1, 9, 31 ) & 22 & 20 & 149 \\
22 & ( 1, 10, 32 ) & 36 & 21 & 159 \\
22 & ( 1, 11, 33 ) & 38 & 43 & 333 \\
22 & ( 1, 12, 34 ) & 39 & 43 & 343 \\
22 & ( 1, 13, 35 ) & 63 & 71 & 663 \\
22 & ( 1, 14, 36 ) & 11711 & 16276 & 150981 \\
22 & ( 1, 15, 37 ) & 151 & 59 & 517 \\
\hline
\end{tabular}
%\caption{dat\label{labdat}}
} 
\end{table}

\begin{table}
\parbox{.45\linewidth}{
        \centering
\begin{tabular}{|c|c|c|c|c|c|}
\hline
$d$ & $(f, g, h)$ & $k$ & $p$ &$m$\\
\hline
22 & ( 4, 15, 37 ) & 1669 & 979 & 8310 \\
22 & ( 4, 16, 38 ) & 96 & 38 & 357 \\
22 & ( 4, 17, 39 ) & 492 & 41 & 367 \\
22 & ( 4, 18, 40 ) & 25 & 22 & 167 \\
22 & ( 4, 19, 41 ) & 25 & 23 & 148 \\
22 & ( 4, 20, 42 ) & 26 & 24 & 153 \\
22 & ( 5, 6, 28 ) & 28 & 126 & 1025 \\
\hline
\end{tabular}
%\caption{dit\label{labdit}}
}
\hfill
\parbox{.45\linewidth}{
        \centering
\begin{tabular}{|c|c|c|c|c|c|}
\hline
$d$ & $(f, g, h)$ & $k$ & $p$ &$m$\\
\hline
22 & ( 8, 18, 40 ) & 28 & 26 & 155 \\
22 & ( 8, 19, 41 ) & 29 & 27 & 160 \\
22 & ( 8, 20, 42 ) & 30 & 28 & 165 \\
22 & ( 9, 10, 32 ) & 114 & 238 & 1848 \\
22 & ( 9, 11, 33 ) & 25 & 21 & 169 \\
22 & ( 9, 12, 34 ) & 291 & 329 & 2760 \\
22 & ( 9, 13, 35 ) & 25 & 22 & 157 \\
\hline
\end{tabular}
%\caption{dat\label{labdat}}
} 
\end{table}

\begin{table}
\parbox{.45\linewidth}{
        \centering
\begin{tabular}{|c|c|c|c|c|c|}
\hline
$d$ & $(f, g, h)$ & $k$ & $p$ &$m$\\
\hline
22 & ( 15, 19, 41 ) & 36 & 34 & 181 \\
22 & ( 15, 20, 42 ) & 37 & 35 & 186 \\
22 & ( 16, 17, 39 ) & 35 & 33 & 174 \\
22 & ( 16, 18, 40 ) & 36 & 34 & 179 \\
22 & ( 16, 19, 41 ) & 37 & 35 & 184 \\
22 & ( 16, 20, 42 ) & 38 & 36 & 189 \\
22 & ( 17, 18, 40 ) & 37 & 35 & 182 \\
\hline
\end{tabular}
%\caption{dit\label{labdit}}
}
\hfill
\parbox{.45\linewidth}{
        \centering
\begin{tabular}{|c|c|c|c|c|c|}
\hline
$d$ & $(f, g, h)$ & $k$ & $p$ &$m$\\
\hline
23 & ( 3, 10, 33 ) & 40 & 45 & 345 \\
23 & ( 3, 11, 34 ) & 41 & 45 & 355 \\
23 & ( 3, 12, 35 ) & 68 & 75 & 685 \\
23 & ( 3, 13, 36 ) & 15034 & 46410 & 411212 \\
23 & ( 3, 14, 37 ) & 159 & 63 & 533 \\
23 & ( 3, 15, 38 ) & 303 & 82 & 723 \\
23 & ( 3, 16, 39 ) & 513 & 596 & 5155 \\
\hline
\end{tabular}
%\caption{dat\label{labdat}}
} 
\end{table}

\begin{table}
\parbox{.45\linewidth}{
        \centering
\begin{tabular}{|c|c|c|c|c|c|}
\hline
$d$ & $(f, g, h)$ & $k$ & $p$ &$m$\\
\hline
23 & ( 6, 19, 42 ) & 27 & 25 & 156 \\
23 & ( 6, 20, 43 ) & 28 & 26 & 161 \\
23 & ( 6, 21, 44 ) & 29 & 27 & 166 \\
23 & ( 7, 8, 31 ) & 157 & 1514 & 12315 \\
23 & ( 7, 9, 32 ) & 322 & 81 & 673 \\
23 & ( 7, 10, 33 ) & 388 & 459 & 3632 \\
23 & ( 7, 11, 34 ) & 446 & 64 & 524 \\
\hline
\end{tabular}
%\caption{dit\label{labdit}}
}
\hfill
\parbox{.45\linewidth}{
        \centering
\begin{tabular}{|c|c|c|c|c|c|}
\hline
$d$ & $(f, g, h)$ & $k$ & $p$ &$m$\\
\hline
23 & ( 11, 16, 39 ) & 29 & 27 & 156 \\
23 & ( 11, 17, 40 ) & 30 & 28 & 161 \\
23 & ( 11, 18, 41 ) & 31 & 29 & 166 \\
23 & ( 11, 19, 42 ) & 32 & 30 & 171 \\
23 & ( 11, 20, 43 ) & 33 & 31 & 176 \\
23 & ( 11, 21, 44 ) & 34 & 32 & 181 \\
23 & ( 12, 13, 36 ) & 27 & 25 & 144 \\
\hline
\end{tabular}
%\caption{dat\label{labdat}}
} 
\end{table}

\begin{table}
\parbox{.45\linewidth}{
        \centering
\begin{tabular}{|c|c|c|c|c|c|}
\hline
$d$ & $(f, g, h)$ & $k$ & $p$ &$m$\\
\hline
24 & ( 1, 8, 32 ) & 20 & 18 & 142 \\
24 & ( 1, 9, 33 ) & 22 & 20 & 153 \\
24 & ( 1, 10, 34 ) & 24 & 22 & 164 \\
24 & ( 1, 11, 35 ) & 39 & 23 & 174 \\
24 & ( 1, 12, 36 ) & 41 & 47 & 363 \\
24 & ( 1, 13, 37 ) & 42 & 47 & 373 \\
24 & ( 1, 14, 38 ) & 59 & 77 & 719 \\
\hline
\end{tabular}
%\caption{dit\label{labdit}}
}
\hfill
\parbox{.45\linewidth}{
        \centering
\begin{tabular}{|c|c|c|c|c|c|}
\hline
$d$ & $(f, g, h)$ & $k$ & $p$ &$m$\\
\hline
24 & ( 4, 8, 32 ) & 39 & 23 & 168 \\
24 & ( 4, 9, 33 ) & 41 & 47 & 351 \\
24 & ( 4, 10, 34 ) & 42 & 47 & 361 \\
24 & ( 4, 11, 35 ) & 1964 & 1979 & 17473 \\
24 & ( 4, 12, 36 ) & 314 & 79 & 715 \\
24 & ( 4, 13, 37 ) & 136 & 193 & 1644 \\
24 & ( 4, 14, 38 ) & 167 & 66 & 555 \\
\hline
\end{tabular}
%\caption{dat\label{labdat}}
} 
\end{table}

\begin{table}
\parbox{.45\linewidth}{
        \centering
\begin{tabular}{|c|c|c|c|c|c|}
\hline
$d$ & $(f, g, h)$ & $k$ & $p$ &$m$\\
\hline
24 & ( 7, 17, 41 ) & 27 & 24 & 177 \\
24 & ( 7, 18, 42 ) & 27 & 25 & 156 \\
24 & ( 7, 19, 43 ) & 28 & 26 & 161 \\
24 & ( 7, 20, 44 ) & 29 & 27 & 166 \\
24 & ( 7, 21, 45 ) & 30 & 28 & 171 \\
24 & ( 7, 22, 46 ) & 31 & 29 & 176 \\
24 & ( 8, 9, 33 ) & 126 & 45 & 352 \\
\hline
\end{tabular}
%\caption{dit\label{labdit}}
}
\hfill
\parbox{.45\linewidth}{
        \centering
\begin{tabular}{|c|c|c|c|c|c|}
\hline
$d$ & $(f, g, h)$ & $k$ & $p$ &$m$\\
\hline
24 & ( 12, 14, 38 ) & 28 & 26 & 151 \\
24 & ( 12, 15, 39 ) & 29 & 27 & 156 \\
24 & ( 12, 16, 40 ) & 30 & 28 & 161 \\
24 & ( 12, 17, 41 ) & 31 & 29 & 166 \\
24 & ( 12, 18, 42 ) & 32 & 30 & 171 \\
24 & ( 12, 19, 43 ) & 33 & 31 & 176 \\
24 & ( 12, 20, 44 ) & 34 & 32 & 181 \\
\hline
\end{tabular}
%\caption{dat\label{labdat}}
} 
\end{table}

\begin{table}
\parbox{.45\linewidth}{
        \centering
\begin{tabular}{|c|c|c|c|c|c|}
\hline
$d$ & $(f, g, h)$ & $k$ & $p$ &$m$\\
\hline
25 & ( 1, 5, 30 ) & 33 & 35 & 293 \\
25 & ( 1, 6, 31 ) & 99 & 143 & 1356 \\
25 & ( 1, 7, 32 ) & 64 & 154 & 1519 \\
25 & ( 1, 8, 33 ) & 20 & 18 & 144 \\
25 & ( 1, 9, 34 ) & 22 & 20 & 155 \\
25 & ( 1, 10, 35 ) & 24 & 22 & 166 \\
25 & ( 1, 11, 36 ) & 39 & 47 & 353 \\
\hline
\end{tabular}
%\caption{dit\label{labdit}}
}
\hfill
\parbox{.45\linewidth}{
        \centering
\begin{tabular}{|c|c|c|c|c|c|}
\hline
$d$ & $(f, g, h)$ & $k$ & $p$ &$m$\\
\hline
25 & ( 3, 21, 46 ) & 1212 & 47 & 425 \\
25 & ( 3, 22, 47 ) & 28 & 25 & 193 \\
25 & ( 3, 23, 48 ) & 28 & 26 & 171 \\
25 & ( 4, 5, 30 ) & 9210 & 9947 & 101760 \\
25 & ( 4, 6, 31 ) & 22 & 20 & 149 \\
25 & ( 4, 7, 32 ) & 24 & 22 & 160 \\
25 & ( 4, 8, 33 ) & 39 & 47 & 341 \\
\hline
\end{tabular}
%\caption{dat\label{labdat}}
} 
\end{table}

\begin{table}
\parbox{.45\linewidth}{
        \centering
\begin{tabular}{|c|c|c|c|c|c|}
\hline
$d$ & $(f, g, h)$ & $k$ & $p$ &$m$\\
\hline
25 & ( 7, 8, 33 ) & 44 & 49 & 369 \\
25 & ( 7, 9, 34 ) & 121 & 204 & 1681 \\
25 & ( 7, 10, 35 ) & 337 & 85 & 729 \\
25 & ( 7, 11, 36 ) & 133 & 48 & 375 \\
25 & ( 7, 12, 37 ) & 278 & 186 & 1515 \\
25 & ( 7, 13, 38 ) & -1 & -1 & 0 \\
25 & ( 7, 14, 39 ) & 1034 & 117 & 975 \\
\hline
\end{tabular}
%\caption{dit\label{labdit}}
}
\hfill
\parbox{.45\linewidth}{
        \centering
\begin{tabular}{|c|c|c|c|c|c|}
\hline
$d$ & $(f, g, h)$ & $k$ & $p$ &$m$\\
\hline
25 & ( 10, 23, 48 ) & 35 & 33 & 192 \\
25 & ( 11, 12, 37 ) & 28 & 506 & 4058 \\
25 & ( 11, 13, 38 ) & 5648 & 5641 & 46223 \\
25 & ( 11, 14, 39 ) & 28 & 25 & 177 \\
25 & ( 11, 15, 40 ) & 28 & 26 & 155 \\
25 & ( 11, 16, 41 ) & 29 & 27 & 160 \\
25 & ( 11, 17, 42 ) & 30 & 28 & 165 \\
\hline
\end{tabular}
%\caption{dat\label{labdat}}
} 
\end{table}

\begin{table}
\parbox{.45\linewidth}{
        \centering
\begin{tabular}{|c|c|c|c|c|c|}
\hline
$d$ & $(f, g, h)$ & $k$ & $p$ &$m$\\
\hline
25 & ( 16, 23, 48 ) & 41 & 39 & 210 \\
25 & ( 17, 18, 43 ) & 37 & 35 & 188 \\
25 & ( 17, 19, 44 ) & 38 & 36 & 193 \\
25 & ( 17, 20, 45 ) & 39 & 37 & 198 \\
25 & ( 17, 21, 46 ) & 40 & 38 & 203 \\
25 & ( 17, 22, 47 ) & 41 & 39 & 208 \\
25 & ( 17, 23, 48 ) & 42 & 40 & 213 \\
25 & ( 18, 19, 44 ) & 39 & 37 & 196 \\
25 & ( 18, 20, 45 ) & 40 & 38 & 201 \\
25 & ( 18, 21, 46 ) & 41 & 39 & 206 \\
25 & ( 18, 22, 47 ) & 42 & 40 & 211 \\
\hline
\end{tabular}
%\caption{dit\label{labdit}}
}
\hfill
\parbox{.45\linewidth}{
        \centering
\begin{tabular}{|c|c|c|c|c|c|}
\hline
$d$ & $(f, g, h)$ & $k$ & $p$ &$m$\\
\hline
25 & ( 18, 23, 48 ) & 43 & 41 & 216 \\
25 & ( 19, 20, 45 ) & 41 & 39 & 204 \\
25 & ( 19, 21, 46 ) & 42 & 40 & 209 \\
25 & ( 19, 22, 47 ) & 43 & 41 & 214 \\
25 & ( 19, 23, 48 ) & 44 & 42 & 219 \\
25 & ( 20, 21, 46 ) & 43 & 41 & 212 \\
25 & ( 20, 22, 47 ) & 44 & 42 & 217 \\
25 & ( 20, 23, 48 ) & 45 & 43 & 222 \\
25 & ( 21, 22, 47 ) & 45 & 43 & 220 \\
25 & ( 21, 23, 48 ) & 46 & 44 & 225 \\
25 & ( 22, 23, 48 ) & 47 & 45 & 228 \\
\hline
\end{tabular}
}
\bigskip
\caption{Exceptional cases: all $(f,g,g+d)$ with $d,f\leq 25$ and $f+g\leq d$
for which neither Theorem \ref{thm:k=3d>1short} nor
Conjecture \ref{conj:except} does apply. Entries $k, p, m$ are
the preperiod length, period length and modulus of periodicity
for $S_{f,g,h}$.
An entry with $k=p=-1$ implies that no periodicity was found after
(at least) 500000 terms; but see also Remark \ref{rmk:hiro}.
\label{tab:except}}
\end{table} 

}

}

\end{document}